\documentclass[11pt,centertags,oneside]{article}
\usepackage{amsmath,amstext,amsthm,amscd,typearea}
\usepackage{amssymb}
\usepackage{a4wide}
\usepackage[mathscr]{eucal}
\usepackage{mathrsfs}
\usepackage{typearea}
\usepackage{charter}
\usepackage{pdfsync}
\usepackage{url}
\usepackage
{hyperref}
\hypersetup{
    colorlinks = 
    true,
    linkcolor={magenta},
    citecolor={magenta},
     }

\usepackage[a4paper,width=16.2cm,top=3cm,bottom=3cm]{geometry}

\numberwithin{equation}{section}

\usepackage[usenames,dvipsnames]{color}

\newtheorem{theorem}{Theorem}[section]

\newtheorem{proposition}[theorem]{Proposition}
\newtheorem{corollary}[theorem]{Corollary}
\newtheorem{lemma}[theorem]{Lemma}

\newtheorem{example}[theorem]{Example}

\newcommand{\cali}[1]{\mathscr{#1}}

\newcommand{\supp}{{\rm Supp}}

\newcommand{\capK}{{\text{\rm cap}}}

\newcommand{\Leb}{\mathop{\mathrm{Leb}}\nolimits}

\newcommand{\dist}{\mathop{\mathrm{dist}}\nolimits}

\newcommand{\ddc}{dd^c}
\newcommand{\dc}{d^c}

\newcommand{\loc}{{loc}}

\newcommand{\dbar}{\overline\partial}
\newcommand{\ddbar}{\partial\overline\partial}

\newcommand{\Cc}{\cali{C}}

\newcommand{\B}{\mathbb{B}}
\newcommand{\C}{\mathbb{C}}
\newcommand{\D}{\mathbb{D}}
\newcommand{\N}{\mathbb{N}}

\newcommand{\R}{\mathbb{R}}
\renewcommand\P{\mathbb{P}}

\title{\bf Moser-Trudinger inequalities and complex Monge-Amp\`ere equation}
\providecommand{\keywords}[1]{\textbf{\textit{Keywords:}} #1}
\providecommand{\subject}[1]{\textbf{\textit{Mathematics Subject Classification 2010:}} #1}

\author{Tien-Cuong Dinh,  George Marinescu, and  Duc-Viet Vu}
\newcommand{\Addresses}{{	
		\textsc{Tien-Cuong Dinh, Department of Mathematics, National University 
of Singapore, 10 Lower Kent Ridge Road, Singapore 119076.}
\noindent
		\par\nopagebreak
		\noindent
		\textit{E-mail address}: \texttt{matdtc@nus.edu.sg}
		\newline
		
		\textsc{George Marinescu, University of Cologne, Department of Mathematics and Computer Science, Division of Mathematics, Weyertal 86-90, 50931, K\"oln,  Germany.}
		\noindent
		\par\nopagebreak
		\noindent
		\textit{E-mail address}: \texttt{gmarines@math.uni-koeln.de}	
		\newline
		
		\textsc{Duc-Viet Vu, University of Cologne, Department of Mathematics and Computer Science, Division of Mathematics, Weyertal 86-90, 50931, K\"oln,  Germany.}
		\noindent
		\par\nopagebreak
		\noindent
		\textit{E-mail address}: \texttt{vuviet@math.uni-koeln.de}	
}}

\date{\today}
\begin{document}
\maketitle
\begin{abstract}  We present a version of the Moser-Trudinger 
inequality in the setting of complex geometry. 
As a very particular case, our result already gives a new Moser-Trudinger 
inequality for functions in the Sobolev space $W^{1,2}$ 
of a domain in $\R^2$.  
We also deduce a new necessary condition 
for the existence of a H\"older continuous solution
of the complex Monge-Amp\`ere equation with right-hand side a given 
measure on a compact K\"ahler manifold.
\end{abstract}
\noindent
\keywords Sobolev space, Moser-Trudinger inequality, Monge-Amp\`ere equation, 
{plurisubharmonic function}, {closed positive current}.

\noindent
\subject{32Uxx}, {32W20}, {46E35}.



\section{Introduction} \label{s:intro}

Moser-Trudinger inequalities are important in Functional Analysis and 
Partial Differential Equations. 
There exist various versions of the Moser-Trudinger inequalities, see 
\cite{ChangYang, Cianchi, Moser, Hoang-MT, Trudinger} 
and the references therein, to cite just a few. 
We just recall here a well-known version of this inequality in the real 
two-dimensional setting. 

Let $\Omega$ be a domain in $\C \approx\R^2$. Let $W^{1,2}(\Omega)$ 
be the Sobolev space of square integrable functions on $\Omega$ 
whose partial derivatives of order one are also square integrable on 
$\Omega$. 
We will denote by $\Leb$ the Lebesgue measure in the Euclidean spaces. 

\begin{theorem}[\cite{Moser}] \label{t:Moser}   
Let $K$ be a compact subset of $\Omega$.  
There exist strictly positive constants $\alpha$ and $c$ such that 
\begin{align}\label{e:MT-1D}
\int_K e^{\alpha |u|^2} d \Leb \le  c
\end{align}
for every $u \in W^{1,2}(\Omega)$ of  $W^{1,2}$-norm at most $1$.
\end{theorem} 

In higher real dimension $n$,  in the present literature, a similar 
inequality holds if we consider the Sobolev space $W^{1,n}$ 
instead of $W^{1,2}$ and the term $|u|^2$ in (\ref{e:MT-1D}) 
is replaced by $|u|^{\frac{n}{n-1}}$. 

Our aim in this paper is  to give a version of the classical Moser-Trudinger inequality 
in the setting of complex geometry. Our result is already new 
in the case of complex dimension one as shown by 
Corollary \ref{c:CRgeneric} below and the comment following it.
Let us first set some notations. 
 
Let $n$ be a positive integer and $\Omega$ a domain in 
$\C^n \approx \R^{2n}$. Let $W^{1,2}(\Omega)$ be the set of square 
integrable functions on $\Omega$ whose partial derivatives of order one are 
also square integrable on $\Omega$.  Let $W^{1,2}_*(\Omega)$ be the 
set of  $u\in W^{1,2}(\Omega)$ such that there exists 
a closed positive $(1,1)$-current $T=T_u$ of bounded mass on $\Omega$ 
with
\begin{align} \label{e:W12-Def}
i \partial u \wedge \overline \partial u \le T.
\end{align}
This functional space was introduced in \cite{DS_decay} 
in the context of complex dynamics and studied in more details in 
\cite{Vigny}, see \cite{DLW,Vu_nonkahler_topo_degree} 
for recent applications to  complex dynamics. 
For $u \in W^{1,2}_*(\Omega),$ put 
$$\|u\|^2_*:= \|u\|_{L^2}^2+
\inf\big\{\|T\|:  \, T\, \text{ satisfies (\ref{e:W12-Def})}\big\}.$$
Note that by the compactness of the space of closed positive currents, 
the last infimum is actually a minimum.
The last formula defines a norm on $W^{1,2}_*(\Omega)$ that becomes 
a Banach space with respect to this norm  (\cite[Proposition 1]{Vigny}). 
In dimension one, we have $W^{1,2}_*(\Omega)=W^{1,2}(\Omega)$.

We set $\dc= \frac{1}{2 \pi}(\overline \partial - \partial)$, 
so that $\ddc= \frac{i}{\pi} \partial \overline \partial$. In the complex 
one-dimensional case, $\ddc$ is simply the Laplace operator. 
We will need some notions from the pluripotential theory. 
We refer to \cite{Bedford_Taylor_82,Chern_Levine_Nirenberg,Demailly_ag,
Klimek,Kolodziej05,Sibony_duke} for an introduction to this topic. 
Recall that a \emph{plurisubharmonic} (or p.s.h.\  for short) function 
$v$ on $\Omega$ satisfies that $\ddc v$ is a closed positive $(1,1)$-current on 
$\Omega$.  A subset $A$ of $\Omega$ is said to be \emph{pluripolar}
if there exists a p.s.h.\ function $v\not\equiv-\infty$ on $\Omega$ such that 
$A \subset \{v= -\infty\}$.  
By a classical  result of Josefson \cite{josefson,Kolodziej05}, 
a locally pluripolar set is pluripolar.
Hence, we don't need to precise the ambient domain when we talk about 
a pluripolar set, see also \cite{Vu_pluripolar}.   

Let $v_1, \ldots, v_n$ be bounded p.s.h.\  functions on $\Omega$. 
It is classical in the pluripotential theory that the intersection 
$\ddc v_1 \wedge \cdots \wedge\ddc v_n$ is well-defined 
and is a positive measure on $\Omega$ having no mass 
on pluripolar subsets of $\Omega$. 
In general, that measure can be \emph{singular} with respect 
to  the Lebesgue measure 
on $\Omega$, see for example \cite{Hiep_holder,Vu_MA}.  
Here is our main result.

\begin{theorem}\label{t:main} 
Let $\Omega$ be a domain in $\C^n$ and $K$ a compact subset of $\Omega$.  
Let $v_1,\ldots, v_n$ be p.s.h.\  functions  which are  H\"older continuous 
of exponent $\beta\in (0,1)$ on $\Omega$. Let  $u \in W^{1,2}_*(\Omega)$. 
Assume that $\|v_j\|_{\cali{C}^\beta} \le 1$ for $1 \le j \le n$  and $\|u\|_* \le 1$. 
Then there exist  strictly positive constants $\alpha$ and $c$ depending on 
$\Omega,K,\beta$ but independent of $u,v_1, \ldots, v_n$ such that
\begin{align}\label{e:main-th}
\int_K e^{\alpha |u|^2} \ddc v_1 \wedge  \cdots \wedge \ddc v_n \le c.
\end{align}
In particular, $u$ belongs to $L^p_{loc}$ with respect to the 
measure $\ddc v_1 \wedge  \cdots \wedge \ddc v_n$ for every 
$p\in[1,\infty)$. 
\end{theorem}

Let us stress two features of Theorem \ref{t:main}. Firstly, 
unlike other known higher dimensional versions of Moser-Trudinger inequalities, 
we get the term $|u|^2$ as in the complex one-dimensional case.
Secondly, \eqref{e:main-th} holds for much more 
general measures than the Lebesgue measure.  Let us make some more comments about 
(\ref{e:main-th}). Firstly, since the elements in $W^{1,2}_*(\Omega)$ 
are \emph{a priori} only measurable functions with respect to Lebesgue 
measures, it is not obvious that the integral in the left-hand side of 
(\ref{e:main-th}) makes sense. To simplify the situation, 
one can consider for the moment $u$ continuous  in (\ref{e:main-th}) 
with values in $\R\cup\{\pm \infty\}$ and the last inequality tells us that 
the integral is bounded uniformly in $u$. Actually, every $u \in W^{1,2}_*(\Omega)$ 
can be represented, in a canonical way, by Borel functions defined 
on $\Omega$ except possibly a pluripolar subset of $\Omega$, 
and such two representatives are equal outside a pluripolar set. 
So the integral in (\ref{e:main-th}) is the integral of one of such 
representatives of $u$ and it is independent of the choice 
of such a representative, see Theorem \ref{t:W12-rep-limit} below and \cite{Vigny}.  Note that the H\"older continuity of $v_j$ in Theorem \ref{t:main} is necessary; see Example \ref{ex-contiu} for a counter-example if $v_j$'s are merely continuous. 

We now present some consequences of Theorem \ref{t:main}.  
Let $(X,\omega)$ be a compact K\"ahler manifold of dimension $n$. 
Let $\varphi$ be a bounded $\omega$-p.s.h.\  function on $X$. 
Recall that $\varphi$ is called $\omega$-p.s.h.\  if $\varphi$ is locally 
the sum of a p.s.h.\  and a smooth function and 
$\ddc \varphi+\omega \ge 0$ in the sense of currents.
The measure $\mu:= (\ddc \varphi+ \omega)^n$ is called 
\emph{a Monge-Amp\`ere measure}. These measures are
central objects of study in complex geometry and pluripotential theory. 
If $\varphi$ is H\"older continuous, 
$\mu$ is called a \emph{Monge-Amp\`ere measure with H\"older potentials}. 
The following is a direct consequence of our main result.

\begin{corollary} \label{c:MA-eq} 
Let $X$ be a compact K\"ahler manifold.  Let $\mu$ be a Monge-Amp\`ere measure with H\"older potentials on $X$. Then there exist  strictly positive constants $\alpha$ and  $c$ such that 
$$\int_X e^{\alpha |u|^2} d \mu \le c$$
for every $u \in W^{1,2}_*(X)$ with $\|u\|_* \le 1.$
\end{corollary}

Here we define $W^{1,2}_*(X)$ in a way similar to that of 
$W^{1,2}_*(\Omega)$.  The last result gives us a necessary 
condition to test whether a given measure is a Monge-Amp\`ere measure 
with H\"older potentials.  We refer to 
\cite{DinhVietanhMongeampere,Lucas_ex,NgocCuong-Kolodziej,Yau1978} 
and the references therein for related results. The readers can also consult \cite[Theorems 2.1 and 4.6]{DiNezzaGuedjLu-MT} for a Moser-Trudinger type inequality for quasi-psh functions of finite energy.  

We give now another application of Theorem \ref{t:main}.    Let $Y$ be a smooth \emph{generic Cauchy-Riemann} 
(real) submanifold of $\Omega$, \emph{i.e.,} given any point $a \in Y$, 
the tangent space of $Y$ at $a$ is not contained in any complex hyperplane 
of the tangent space of $\Omega$ at the point. 
The simplest example is $Y:= \R^n \cap \Omega$, 
where $\R^n \approx \R^n + i \,0 \hookrightarrow \C^n:= \R^n + i \R^n$. 
Let $K$ be a compact subset of $Y$.  Since  $\C^n \subset \P^n$ 
(the complex projective space of dimension $n$), using \cite{Vu_MA}, 
we see that the restriction of  a Lebesgue measure of $Y$ to $K$ is a 
Monge-Amp\`ere measure with H\"older potentials on $\Omega$ 
(here by a Lebesgue measure, we mean the volume with respect to a 
smooth Riemannian metric on $Y$). Hence, Theorem \ref{t:main} 
immediately gives us  the following result.   

\begin{corollary}  \label{c:CRgeneric} 
Let $\Omega$ be a domain in $\C^n$ and 
$Y$ be a smooth generic Cauchy-Riemann submanifold of $\Omega$. 
Let $K$ be a compact subset of $Y$ and let $\Leb$ denote a  fixed 
Lebesgue measure on $Y$.  
Then there exist strictly positive constants $\alpha$ and $c$ such that 
$$\int_K e^{\alpha |u|^2} d \Leb \le c$$
for every $u \in W^{1,2}_*(\Omega)$ with $\|u\|_* \le 1$.
\end{corollary}

Observe that the last result is already new even if we apply it 
to the simplest situation where  $\Omega= \D$ is the unit disc in $\C$ 
and $K \Subset \R \cap \D$.  
We also obtain from the last result applied to $Y=\Omega$ the following corollary.

\begin{corollary} \label{cor-Lp-CV}
Let $(u_k)_k$ be a bounded sequence in $W^{1,2}_*(\Omega)$ 
converging to a function $u\in W^{1,2}_*(\Omega)$ in the sense of currents.
Then $u_k$ converges to $u$ in $L^p_\loc$ for every $1\leq p<+\infty$.
\end{corollary}

In the next section, we present some facts about the space $W^{1,2}_*$. In the last section, we prove the main result and Corollary \ref{cor-Lp-CV}. Our proof of Theorem \ref{t:main} consists of two steps. Firstly,  we use the H\"older continuity of $v_1, \ldots, v_n$ and arguments similar to those in \cite{DVS_exponential,DinhVietanhMongeampere} to reduce the question to the case where $v_1, \ldots, v_n$ are smooth. In the second step, we use slicing of currents to make a reduction to a lower dimension case and then we apply Theorem \ref{t:Moser}.
The case with $v_j$ smooth is enough to obtain Corollary \ref{cor-Lp-CV}.

\medskip

\noindent
\textbf{Acknowledgments.} The first author is supported by the NUS 
and MOE grants AcRF Tier 1 R-146-000-319-114 and MOE- T2EP20120-0010.   
The second and third authors are 
partially supported by the DFG funded project SFB TRR 191 
(Project-ID 281071066 -- TRR 191).


\section{Properties of functions in the complex Sobolev space} \label{s:Sobolev}

Let $\Omega$ be a domain in $\C^n$. 
We present properties of the space $W^{1,2}_*(\Omega)$.   To simplify the notation, we write $W^{1,2}_*$  instead of $W^{1,2}_*(\Omega)$ if no confusion arises. 
We will use some basic results in pluripotential theory and refer the reader to  \cite{Bedford_Taylor_82, Chern_Levine_Nirenberg, Demailly_ag,Kolodziej05,Sibony_duke} for details. 

\smallskip\noindent
{\bf Standard regularization.} The following approximation of functions and 
currents will be used several times in this paper.
Let $\chi$ be a smooth  nonnegative cut-off  radial function with support 
in a ball $\B(0,R)$ of center 0 and radius $R$ in $\C^n$ such that 
$\int_{\C^n} \chi \, d \Leb =1$. For every real number $\varepsilon>0$, 
put $\chi_\varepsilon(x):= \varepsilon^{-2n} \chi(\varepsilon^{-1}x)$. 
We have $\int_{\C^n} \chi_\varepsilon \, d \Leb =1$ and the support of 
$\chi_\varepsilon$ is contained in the ball of center 0 and radius $R\varepsilon$. 
For every function $u\in L^1_\loc(\Omega)$, we define the convolution
$$u_\varepsilon(x):= u* \chi_\varepsilon (x) = 
\int_{\C^n} u(x- y) \chi_\varepsilon(y) \,d\Leb(y).$$
This is a smooth function defined on $\Omega_\varepsilon:=
\{x: \dist(x, \partial \Omega) > R \varepsilon\}$. 
We call it {\it a standard regularization} of $u$.
Since $\chi$ is nonnegative radial, if $u$ is p.s.h., 
$u_\varepsilon$ is smooth, p.s.h.\  and decreases to $u$ when 
$\varepsilon$ decreases to 0, thanks to the submean inequality. 

Let $T$  be a closed positive $(1,1)$-current on $\Omega$.
Write $T= \ddc \varphi$ locally and define 
$T_\varepsilon:= \ddc \varphi_\varepsilon$, where $\varphi_\varepsilon$ 
is the standard regularization of $\varphi$. Observe that $T_\varepsilon$ is 
independent of the choice of $\varphi$ (because if $\varphi'$ is another local potential of $T$, then $\ddc (\varphi'- \varphi)=0$, hence we obtain $\ddc \varphi'_\epsilon= \ddc\varphi_\epsilon$). 
Therefore, we obtain a closed positive $(1,1)$-current 
$T_\varepsilon$ defined on $\Omega_\varepsilon$ that we also call 
{\it a standard regularization} of $T$.

\smallskip\noindent
{\bf Wedge-product of currents and continuity.} 
Let $R$ be a closed positive current on $\Omega$. 
Recall that if $v$ is a bounded p.s.h.\  function, then $\ddc v \wedge R:= \ddc (v R)$ is a closed positive current. Hence, for bounded p.s.h.\  functions $v_1, \ldots, v_l$, we can define inductively $\ddc v_1 \wedge \ldots \wedge \ddc v_l \wedge R$. It is well-known that 
both $v_1 \ddc v_2\wedge \ldots\wedge \ddc v_l\wedge R$ and $\ddc v_1 \wedge \ldots \wedge \ddc v_l \wedge R$
 depend  continuously on $v_1, \ldots, v_l$ by taking sequences of p.s.h.\  functions decreasing to $v_1, \ldots, v_l$. So, we can apply this property for the standard regularization of $v_j$ described above. 

Let $v_1$ and $v_2$ be  bounded p.s.h.\  functions on $\Omega$. If $A>0$ is a large enough constant, we have $v_j+A\geq 0$ and hence $(v_j+A)^2$ and $(v_1+v_2+2A)^2$ are p.s.h.\  functions. It follows that $(v_1-v_2)^2$ is the difference of two  bounded p.s.h.\  functions because we can write
$$(v_1-v_2)^2= [2(v_1+A)^2+2(v_2+A)^2]-(v_1+v_2+2A)^2.$$
This, together with the identity $\ddc v^2 = 2(dv\wedge \dc v+ v\ddc v)$, allow us to define 
\begin{align}\label{e:d-dc-def}
d(v_1-v_2) \wedge \dc (v_1 -v_2) \wedge R:= {1\over 2} \ddc (v_1 - v_2)^2 \wedge R -  (v_1- v_2) \ddc (v_1- v_2) \wedge R.
\end{align} 

\smallskip\noindent
{\bf Capacity and convergence in capacity.} 
Let $K$ be a Borel subset of $\Omega$. Recall that the \emph{capacity} of $K$ in $\Omega$ is the quantity 
$$\capK (K,\Omega):= \sup \bigg\{ \int_K (\ddc v)^n: \, 0 \le v \le 1 \quad 
\text{p.s.h.\  on }\, \Omega \bigg\}.$$
This notion was introduced in \cite{Bedford_Taylor_82}.  
Every pluripolar set in $\Omega$ is of zero capacity in $\Omega$. 
Recall that  $\Omega$ is called \emph{hyperconvex} if there exists 
a continuous p.s.h.\ function $\rho: \Omega \to (-\infty, 0)$ such that $\{\rho < c\}$ is relatively compact in $\Omega$ for every constant $c<0$. Examples of such  domains are  balls in $\C^n$.  
If $\Omega$ is hyperconvex, then  a subset $A$ of $\Omega$ is pluripolar if and only if  $\capK^* (A,\Omega)=0$, where 
$$\capK^*(A, \Omega):= \inf \big\{ \capK(U, \Omega): \, A \subset U \subset \Omega, \, U \text{ open}\big\},$$
see \cite{Bedford_Taylor_82,Kolodziej05}.  

 Let $u_k$ be a Borel function defined everywhere on $\Omega$ except on a pluripolar subset of $\Omega$ for $k\in \N$. We say that $(u_k)_{k\in \N}$ is a \emph{Cauchy sequence with respect to capacity} if for every constant $\delta>0$,  every open subset $U$ in $\Omega$ and every compact set $K$ in $U$, we have  
\begin{align} \label{e:cap-cauchy-def}
\lim_{N \to \infty}\sup_{\{k, l\ge N\}} \capK \big(K\cap\{|u_k -u_l|\ge \delta\},U\big)=0.
\end{align}
 Similarly, given a Borel function $u$ defined on $\Omega$ except maybe on a pluripolar set,  we say that \emph{$u_k$ converges to $u$ in capacity as $k \to \infty$} or $u$ is \emph{a capacity limit of $(u_k)_k$} if for every open set $U \subset \Omega$, every compact set $K \Subset U$, and every constant $\delta>0$, we have 
\begin{align} \label{e:cap-limit-def}
\capK\big(K\cap \{|u_k- u|\ge \delta\}, U\big) \to 0
\end{align}
 as $k \to \infty$, see \cite{Kolodziej05} for details.  One can check that capacity limits of a given sequence only differ on pluripolar sets.   Notice also that if $u_k$ converges to some function $u$ in capacity as $k \to \infty$, then $(u_k)_k$ is a Cauchy sequence with respect to capacity. The above two notions are local because 
$$\capK(K\cup K', \Omega) \le \capK(K, \Omega)+ \capK(K', \Omega)$$ 
for $K \subset \Omega, K' \subset \Omega$  and $\capK(K, \Omega) \le \capK(K, \Omega')$ when $K\subset \Omega' \subset \Omega$.

\smallskip\noindent
{\bf Two other notions of convergence.}
Let  $u_k \in W^{1,2}_*$ for $k \in \N$ and $u \in W^{1,2}_*$.
We say that $u_k \to u$ in the \emph{weak} topology of $W^{1,2}_*$ if 
$u_k \to u$  in  the sense of distributions and $\|u_k\|_*$ is uniformly bounded. 
Since $W^{1,2}_*$ is continuously embedded in $W^{1,2}$, 
by Rellich's theorem, we have $u_k\to u$ in $L^{2n\over n-1}_\loc$ 
(or $L^p_{loc}$ for  every $1\leq p<\infty$ when $n=1$). In particular, 
we have $u_k\to u$ in $L^2_\loc$. Note that Corollary \ref{cor-Lp-CV} in Introduction gives a much stronger property.
Assume that $u_k\to u$ weakly in $W^{1,2}_*$ as above. 
Assume also that $i\partial u_k\wedge \dbar u_k\leq T_k$ 
for some closed positive $(1,1)$-current $T_k$ converging to a current $T$. 
Then we have $i\partial u\wedge \dbar u\leq T$, see \cite[p. 251]{Vigny}. 

We say that $u_k \to u$ \emph{nicely}   
if $u_k \to u$ weakly in $W^{1,2}_*$ and  for every $x \in \Omega$, 
there exist an open neighbourhood $U_x$ of $x$ and a p.s.h.\  
function $\varphi_k$ on $U_x$ such that 
\begin{align*} 
i \partial u_k \wedge \overline \partial u_k \le \ddc \varphi_k
\end{align*}
 for every $k$ and  $\varphi_k$ decreases to some p.s.h.\  function on $U_x$. 
 
\medskip
 
For $K \subset \Omega$ and $R$ a current on $\Omega$, 
we denote by $\|R\|_K$  the mass of $R$ on $K$. 
We will need the following important estimates, see 
\cite{Chern_Levine_Nirenberg, Kolodziej05}. 

\begin{lemma} \label{l:CLN+}  
Let $K \Subset \Omega$ be a compact set. Let $v_1, \ldots, v_m$ be 
bounded p.s.h.\  functions on $\Omega$ and $\varphi$ another p.s.h.\  function on 
$\Omega$. Let $R$ be a closed positive $(p,p)$-current on $\Omega$ with 
$0\leq p\leq n-1$. Then, there exists a positive constant $c$ depending 
only on $K$ and $\Omega$ such that
 $$\|\ddc v_1 \wedge \cdots \wedge \ddc v_m \wedge R 
 \|_K \le c \|v_1\|_{L^\infty(\Omega)} \cdots \|v_m\|_{L^\infty(\Omega)} 
 \|R\|_\Omega,$$
for $1\leq m\leq n-p$ and 
$$\|d (v_1-v_2) \wedge \dc (v_1- v_2) \wedge R \|_K \le 
c \big( \|v_1\|_{L^\infty(\Omega)}+ 
\|v_2\|_{L^\infty(\Omega)}\big)\|v_1-v_2 \|_{L^\infty(\Omega)}  
\|R\|_\Omega,$$
and  for every constant $N>0$, 
$$\capK\big(\{\varphi \le -N\} \cap K, \Omega\big) \le 
c N^{-1}\|\varphi\|_{L^1(\Omega)}.$$
\end{lemma} 

\proof For the first  and third estimates, we refer to 
\cite[p.\ 8 and Proposition 1.10]{Kolodziej05}, see also  
\cite{Chern_Levine_Nirenberg}. 
The proof of the first estimate is based on an induction on $m$ and 
the use of integration by parts. We use the same techniques together 
with \eqref{e:d-dc-def} in order to get the second estimate.
We give here the details as we will need them later.

Since the problem is local, we can assume that $\Omega$ is the unit ball. 
By subtracting from $v_1, v_2$ a same constant, we can assume that  
$$v_j\leq -{1\over 2} \max\big(\|v_1\|_{L^{\infty}(\Omega)}, 
\|v_2\|_{L^{\infty}(\Omega)}\big).$$ 
Let $A>0$ be equal to a large enough constant (depending on $K$) 
times  $\|v_1\|_{L^\infty(\Omega)}+ \|v_2\|_{L^\infty(\Omega)}$. 
Then we replace $v_j$ by $v_j':=\max(v_j,A(\|x\|^2-1))$, where $x$ 
denotes the standard coordinates in $\C^n$. We can choose $A$ large enough 
so that $v_j'=v_j$ on $K$ and $v_j'=A(\|x\|^2-1)$ outside a fixed ball $L$ 
such that $K\Subset L\Subset\Omega$. Since $|v_1'-v_2'|\leq |v_1-v_2|$, 
this step doesn't affect our problem. In this way, we can assume for simplicity that 
$v_1=v_2=A(\|x\|^2-1)$ outside $L$ and hence $v_1-v_2=0$ there. 
Let $\omega$ denote the standard K\"ahler form on $\C^n$. 
We have using \eqref{e:d-dc-def}
\begin{eqnarray*}
\lefteqn{\|d (v_1-v_2) \wedge \dc (v_1- v_2) \wedge R \|_K \  \leq \  
\int_\Omega d (v_1-v_2) \wedge \dc (v_1- v_2) \wedge R\wedge \omega^{n-p-1}} \\
& = &  {1\over 2} \int_\Omega \ddc (v_1 - v_2)^2 \wedge R\wedge \omega^{n-p-1}  
- \int_\Omega (v_1- v_2) \ddc (v_1- v_2) \wedge R \wedge \omega^{n-p-1}.
\end{eqnarray*}
As $v_1-v_2=0$ on $\Omega\setminus L$, the first integral in the last line 
vanishes by integration by parts. The second one is bounded by
$$\|v_1-v_2\|_{L^\infty(\Omega)} \big[\|\ddc v_1\wedge R\|_{L}+
\|\ddc v_2\wedge R\|_{L}\big].$$
We obtain the second estimate in the lemma by applying the first one to 
$L$ instead of $K$.
\endproof
 
We have the following elementary property of Cauchy sequence with 
respect to capacity.

\begin{lemma} \label{l:Cauchy-limit} 
Let $(u_k)_k$ be a sequence of continuous functions on $\Omega$. 
Assume that $(u_k)_k$ is a Cauchy sequence with respect to capacity. 
Then there exists a Borel function $u_\infty$ defined everywhere on 
$\Omega$ except on a pluripolar set such that 

$(i)$  $u_k$ converges to $u_\infty$ in capacity; if $u_k$ 
converges to another function $u_\infty'$ in capacity, then 
$u_\infty'=u_\infty$ outside a pluripolar set;

$(ii)$ there exists a sequence $(j_k)_{k} \subset \N$ converging to $\infty$ 
such that  $u_{j_k} \to u_\infty$ pointwise except on a pluripolar set as $k \to \infty$;

$(iii)$  if  $(j'_k)_{k} \subset \N$ is another sequence converging to $\infty$ 
such that $u_{j'_k}$ converges pointwise to some function $u'_\infty$ 
outside a pluripolar set, then $u'_\infty= u_\infty$ outside a pluripolar set.  
\end{lemma}

\proof   
Assume (ii) for the moment. We explain how to get (i) and (iii).  
The second assertion in (i) is clear. Let $(u_{j_k})_k$ be a subsequence 
of $(u_k)_k$ such that $u_{j_k} \to u_\infty$ pointwise except on a pluripolar set.
Let $U$ be an open subset of $\Omega$ and  $K$ a compact subset of $U$. 
Let $\varepsilon>0 $ be a constant. By (\ref{e:cap-cauchy-def}), 
there exists $N\in \N$ big enough so that 
$$ \capK \big(K\cap\{|u_{l_1} -u_{l_2}|\ge \delta/2\},U\big) \le \varepsilon$$ 
for every $l_1,l_2 \ge N$. Applying the last inequality to $l_2= j_k$ 
and letting $k \to \infty$ give
$$ \capK \big(K\cap\{|u_{l_1} -u_\infty |\ge \delta\},U\big) \le \varepsilon$$ 
for every $l_1 \ge N$. This implies that $u_{l_1} \to u_\infty$ in capacity as 
$l_1\to \infty$. Hence, (i) follows. Let $(j'_k)_k$ and $u'_\infty$ be as in (iii). 
Then by the above arguments, we get $u_k \to u'_\infty$ in capacity. Hence, 
$u'_\infty= u_\infty$ outside a pluripolar set. 

It remains to prove (ii).  Let $(\Omega_s)_{s\in\N}$ 
be a countable covering of $\Omega$ by  open balls. 
As observed above, $\Omega_s$ is hyperconvex for every $s$. 
Let $(\Omega'_s)_{s\in\N}$ be another covering of $\Omega$ 
by open balls  such that $\Omega'_s \Subset \Omega_s$ for every $s$.
Fix $s \in\N$.  Let $\delta>0$ be a constant. 
By (\ref{e:cap-cauchy-def}), there exists a sequence 
$(j^s_k) \to\infty$ such that  for every $k$, we have
$$\capK (E_k^s\cap  \Omega'_s, \Omega_s) \le 
\capK (E_k^s\cap \overline \Omega'_s, \Omega_s) \le \delta/2^{k},$$
where $E_k^s:= \big\{|u_{j^s_k}- u_{j^s_{k+1}}| >2^{-k} \big\}$ 
which is an open set.  Hence,  for $ E^s_\delta:= \cup_{k=1}^\infty E^s_k,$ 
the sequence $(u_{j^s_k})_k$ converges uniformly on 
$\Omega'_s \backslash E^s_\delta$. 

Observe that $\capK(E_\delta^s \cap\Omega'_s, \Omega_s) \le \delta$.
For $\delta=1/m$, by a diagonal argument,  we can assume that $u_{j^s_k}$ 
converges uniformly on $\Omega'_s \backslash E^s_{1/m}$ for every $m$. 
Hence, $u_{j^s_k}$ converges pointwise on  
$\Omega'_s \backslash (\cap_m E^s_{1/m})$. Since 
$$\capK (E^s_{1/m}\cap \Omega'_s, \Omega_s)\le  1/m$$ for every $m$, 
 we obtain that  $ \capK^* (\cap_m E_{1/m}^s\cap\Omega'_s, \Omega_s) =0$. 
 Hence, $\cap_m E_{1/m}^s \cap \Omega'_s$ is pluripolar.  
 This implies that $u_{j^s_k}$ converges pointwise on $\Omega'_s$ 
 except on a pluripolar set. 

Applying the above arguments to $s=1,2, \ldots$ and using 
a diagonal argument again, we obtain a sequence $(j_k)_k \subset \N$ 
converging to $\infty$ such that $u_{j_k}$ converges pointwise on 
$\Omega'_s$ except on a pluripolar set  for every $s$. Thus, (ii) follows.
This finishes the proof. 
\endproof  
 
The following result provides a good regularization for functions in 
$W^{1,2}_*$, see \cite{Vigny}. 

\begin{lemma} \label{l:nice-CV}  
Let $u$ be a function in $W^{1,2}_*(\Omega)$ and let $u_\varepsilon$ 
be the standard regularization of $u$ as above. Then 
we have  $\|u_\varepsilon\|_{*,\Omega_\varepsilon} \le \|u\|_{*,\Omega}$, 
$\|u_\varepsilon\|_{L^\infty(\Omega_\varepsilon)} \le  \|u\|_{L^\infty(\Omega)}$, 
and if $u\ge 0$, we have $u_\varepsilon \ge 0$ for every $\varepsilon$. 
Moreover, for every sequence $(\varepsilon_k)_k$ decreasing to $0$,
we have $u_{\varepsilon_k} \to u$ nicely in $W^{1,2}_*(\Omega')$ 
for every open set $\Omega'\Subset\Omega$.
\end{lemma}

\proof 
It is clear from the definition of $u_\varepsilon$ that 
$\|u_\varepsilon\|_{L^\infty(\Omega_\varepsilon)} \le  
\|u\|_{L^\infty(\Omega)}$ and  $u_\varepsilon \ge 0$ if $u\ge 0$. 
It is also clear that
$\|u_\varepsilon\|_{L^2(\Omega_\varepsilon)}\leq \|u\|_{L^2(\Omega)}$ 
and $u_\varepsilon \to u$ in $L^2(\Omega')$ for every $\Omega'\Subset\Omega$.
Let  $T$  be a closed positive $(1,1)$-current on $\Omega$ with $\|T\|$ 
minimal such that (\ref{e:W12-Def}) holds. 
Denote by $T_\varepsilon$ the standard regularization of $T$.  
By \cite[Lemma 5]{Vigny}) (see also the proof of Lemma \ref{l:Sobolev-mean-T} below), 
we have 
$i \partial u_\varepsilon \wedge \overline \partial u_\varepsilon \le T_\varepsilon$
on $\Omega_\varepsilon.$  So we deduce that 
$u_\varepsilon \in W^{1,2}_*(\Omega_\varepsilon)$ and its 
$*$-norm bounded by $\|u\|_{*,\Omega}$. We conclude that
$u_{\varepsilon_k} \to u$ nicely in $W^{1,2}_*(\Omega')$
as we have seen that if we write $T= \ddc \varphi$, then $\varphi_\varepsilon$ decreases to $\varphi$ 
when $\varepsilon$ decreases to 0. 
This finishes the proof.     
\endproof

Note that Lipschitz functions belong to $W^{1,2}_*$. 
The following result shows that
$W^{1,2}_*$ is closed under basic operations on functions 
and allows us to produce functions in this space, see \cite{DS_decay}. 

\begin{lemma} \label{l:W12-operations} 
Let $\tau: \R \to \R$ be a Lipschitz function and 
$u\in W^{1,2}_*$. Define $u^\pm:=\max\{\pm u,0\}$. 

$(i)$ We have $\tau(u) \in W^{1,2}_*$ and 
$\|\tau(u)\|_* \le c(|\tau(0)|+ \|u\|_*)$ for some constant $c>0$ 
independent of $u.$ In particular, we have 
$u^+, u^-, |u| \in W^{1,2}_*$ and 
$\max\{u_1, u_2\} \in W^{1,2}_*$ if $u_1, u_2 \in W^{1,2}_*.$ 
 
$(ii)$ If $u_k \to u$ weakly in $W^{1,2}_*$, then $\tau(u_k) \to \tau(u)$ 
weakly in $W^{1,2}_*$.  If $u_k \to u$ nicely in $W^{1,2}_*$, 
then $\tau(u_k) \to \tau(u)$ nicely  in $W^{1,2}_*$.
 
$(iii)$ Assume that $\Omega$ is bounded.  Let $v$ be a p.s.h.\  function 
on an open neighborhood $\widetilde\Omega$ of $\overline \Omega$ 
such that $0 \le v \le 1$. Then $v$ belongs to $W^{1,2}_*(\Omega)$ 
and $\|v\|_*$ is bounded by a constant depending only on $\Omega$ 
and $\widetilde\Omega$.
\end{lemma}

\proof 
As in \cite[Prop. 4.1 and Lemma 4.2]{DS_decay}, 
we easily obtain (i) and (ii) 
using that $|\tau(t)|\le |\tau(0)|+ A |t|$ and 
$i\partial \tau(u)\wedge \dbar \tau(u) \leq A^2 i\partial u\wedge \dbar u$ 
if $\tau$ is $A$-Lipschitz. We also used here that the maps 
$t \mapsto t^+,t^-, |t|$ are $1$-Lipschitz and 
$\max\{u_1, u_2\}= (u_1- u_2)^{+}+ u_2$.
The assertion (iii) is a direct consequence of (i) by using 
$i\partial v\wedge \dbar v\leq i\ddbar v^2$ and by observing that 
$v^2$ is a p.s.h function. 
\endproof

We will need the following estimates, see also Proposition \ref{p:W12-MA-general} 
below. We note that results related to  the inequality \eqref{l:W12-L2-MA-1}  below
were proved in \cite{Okada, Vigny}.
  
\begin{lemma} \label{l:W12-d-dc-L2} 
Let $u \in W^{1,2}_* \cap \cali{C}^0(\Omega)$ with $\|u\|_* \le 1$ 
and let $v_1, \ldots, v_n$ be p.s.h.\  functions on $\Omega$ with values in $[0,1]$.
Let $K$ be a compact subset of $\Omega$. 
Then there is a constant $c>0$ depending only on $K$ and $\Omega$ such that 
\begin{align}\label{l:W12-L2-MA-1}
\int_K u^2 \ddc v_1 \wedge \cdots \wedge \ddc v_n  \le c,
\end{align}
and if $i \partial u \wedge \overline \partial u \le \ddc \varphi$ 
for some p.s.h.\  function $\varphi$ on $\Omega$ such that $0\leq\varphi\leq 1$, 
then
$$
\int_K u^2 \ddc v_1 \wedge \cdots \wedge \ddc v_n    \le 
c\, \bigg( \int_\Omega u^2 (\ddc \varphi+\omega)^n\bigg)^{1/2^n},
$$
where  $\omega$ is the standard K\"ahler form on $\C^n$.
\end{lemma}

\proof  By regularization, we can assume that $u$ is smooth.
The point here is that the constant $c$ is independent of $u$.   
Let $T$ be a closed positive $(1,1)$-current so that $\|T\| \le 1$  and 
$$i \partial u \wedge \overline \partial u \le T.$$
In order to get  \eqref{l:W12-L2-MA-1} it is enough to prove by induction on $0 \le l \le n$ that 
$$\int_K u^2  \ddc v_1 \wedge \cdots \wedge \ddc v_l \wedge \omega^{n-l}\leq c$$
for some constant $c$ depending only on $K$ and $\Omega$. 
The case where $l=0$ is clear, see the beginning of this section. 
Assume this property for $l-1$ instead of $l$ and for every $K$. We now prove it for $l$.

Since the problem is local, we can assume that $\Omega$ 
is the unit ball in $\C^n$. Fix a compact set $K$ in $\Omega$. 
As in the proof of Lemma \ref{l:CLN+}, we can assume that  all 
$v_j$ are smooth outside a compact set $L$ with $K\Subset L\Subset\Omega$ 
and $\|v_j\|_{\Cc^2(\Omega\setminus L)} \leq 1$. 
Fix a smooth function $0\leq \tau\leq 1$ with $\tau=1$ on a 
neighbourhood of $L$ such that $\tau$ is supported by a compact set $K'$ in $\Omega$. 
Define 
$$I_{\rm max}:= \sup_{v'_j} \int_\Omega \tau u^2  \ddc v'_1 
\wedge \cdots \wedge \ddc v'_l \wedge \omega^{n-l},$$
where the supremum is taken over all  p.s.h.\  functions $v'_j$ 
with $0 \le v'_j \le 1$ and $\|v'_j\|_{\Cc^2(\Omega\setminus L)} \leq 1$. 
Define
\begin{align}\label{ine-auynaplL2norm}
I:= \int_\Omega \tau u^2  \ddc v_1 \wedge \cdots \wedge 
\ddc v_l \wedge \omega^{n-l} \quad \text{and} \quad 
R:= \ddc v_2 \wedge \cdots \wedge \ddc v_l \wedge \omega^{n-l}.
\end{align}
Since $u$ is smooth, we can perform integration by parts to obtain
$$I= -\int_\Omega u^2 d \tau \wedge \dc v_1 \wedge R-
2\int_\Omega \tau u d u  \wedge \dc v_1  \wedge R.$$
Denote by $I_1, I_2$ the first and second integral 
in the right-hand side of the last equality. 

Observe that $d \tau \wedge \dc v_1 \wedge R$ 
is a bounded form because $d\tau$ vanishes outside 
$K'\setminus L$ and $\|v_j\|_{\Cc^2(K'\setminus L)}\leq 1$. 
Therefore, $|I_1|$ is bounded by a constant because $u$ has bounded 
$L^2(\Omega)$-norm. 
For $I_2$, since $i\partial u\wedge \dbar u\leq T$, we obtain
$$d u \wedge \dc u \wedge R=
\pi^{-1} i\partial u \wedge \overline \partial u \wedge R  \le R\wedge T.$$
This, together with the Cauchy-Schwarz inequality and Lemma 
\ref{l:CLN+}, give 
\begin{align*}
|I_{2}|  &\le \bigg(\int_\Omega \tau d u  \wedge 
\dc u \wedge R\bigg)^{1/2}\bigg(\int_\Omega \tau u^2 d v_1  
\wedge \dc v_1 \wedge R\bigg)^{1/2} \\
&\le \bigg(\int_\Omega \tau R\wedge T \bigg)^{1/2} 
\bigg(\int_\Omega \tau u^2  \ddc v_1^2 \wedge R\bigg)^{1/2} \\
&\le\bigg(\int_\Omega \tau u^2  \ddc v_1^2 \wedge R\bigg)^{1/2} \lesssim  
(I_{\rm max})^{1/2}
\end{align*}
(for the last inequality, we used that $v_1^2$ is p.s.h.\ , $0\leq v_1^2\leq 1$ 
and $\|v_1^2\|_{\Cc^2(\Omega\setminus L)}$ is bounded). 
It follows  that 
$|I| =|I_1+2I_2|\lesssim 1+ (I_{\rm max})^{1/2}$
for every $v_j$ as above. Therefore, we deduce from the definition of 
$I_{\rm max}$ that $I_{\rm max} \lesssim 1+ (I_{\rm max})^{1/2}$.
Thus, $I_{\rm max}$ is bounded by a constant 
and the inequality \eqref{l:W12-L2-MA-1} is proved. 

We now prove the second inequality in the lemma. 
Using \eqref{l:W12-L2-MA-1}, we can reduce slightly $\Omega$ 
and assume that 
$$Q:= \int_\Omega u^2 (\ddc \varphi+\omega)^n$$
is bounded by a constant. It follows that $Q\lesssim Q^{1/2^l}$ for 
$0\leq l\leq n$. We can now follow the main lines of the proof of
 \eqref{l:W12-L2-MA-1} but we replace each 
 $\omega$ by $\ddc\varphi+\omega$. Denote by $I', R', I'_1, I'_2$ 
 the quantities defined as $I, R, I_1, I_2$ but we replace $\omega$ by 
 $\ddc\varphi+\omega$. 
With the same arguments, we get $|I_1'|\lesssim Q\lesssim Q^{1/2^l}$. 
The estimate of $|I_2'|$ is slightly different.
By Cauchy-Schwarz inequality, the inequality 
$i du \wedge \dc u \leq \ddc \varphi$, Lemma \ref{l:CLN+} 
and the induction hypothesis, we get
\begin{align*}
|I'_{2}|  &\le \bigg(\int_\Omega \tau u^2 d u  \wedge \dc u 
\wedge R'\bigg)^{1/2}\bigg(\int_\Omega \tau d v_1  \wedge 
\dc v_1 \wedge R'\bigg)^{1/2}\\
\nonumber
&\lesssim \bigg(\int_\Omega \tau u^2  \ddc \varphi \wedge R' \bigg)^{1/2} 
\lesssim (Q^{1/2^{l-1}})^{1/2},
\end{align*}
where we bound $\int_\Omega \tau d v_1  \wedge 
\dc v_1 \wedge R'$ by using Lemma \ref{l:CLN+}  and the fact that $0 \le v_j, \phi \le 1$.  We conclude that $|I'|=|I_1'+2I_2'|\lesssim Q^{1/2^{l}}$ which ends the proof of the lemma.
\endproof

\begin{lemma} \label{l:W12-L2-limit-0}
Let $(u_k)_k\subset W^{1,2}_*\cap \cali{C}^0(\Omega)$ be a sequence 
converging to $0$ weakly in $W^{1,2}_*$. 
For each $j=1,\ldots,n$, let $(v_{j,k})_k$ be a sequence of bounded p.s.h.\ functions 
on $\Omega$ decreasing to some bounded p.s.h.\  function $v_j$ when $k$ tends to infinity. 
Then we have for every compact set $K$ in $\Omega$
$$\lim_{k \to \infty}\int_K  |u_k| \ddc v_{1,k} \wedge \cdots \wedge \ddc v_{n,k}=0.$$
In particular, if we assume moreover that $(u_k)_k$ is uniformly bounded, then 
$$\lim_{k \to \infty}\int_K  u_k^2 \ddc v_{1,k} \wedge \cdots \wedge \ddc v_{n,k}=0.$$
\end{lemma}
\proof
By Lemma \ref{l:CLN+}, the mass of 
$\ddc v_{1,k} \wedge \cdots \wedge \ddc v_{n,k}$ on $K$ 
is bounded by a constant. Therefore, the second assertion of the lemma 
is a direct consequence of the first one. We prove now the first assertion. 

By replacing $u_k$ with $|u_k|$, we can assume that $u_k\geq 0$, 
see Lemma \ref{l:W12-operations}. Then, using the standard regularization, 
we can assume that $u_k$ is smooth.
Since the problem is local, as in the proof of Lemma \ref{l:CLN+}, we can assume that 
$\|v_{j,k}\|_{\Cc^2(\Omega\setminus L)}\leq 1$ for some compact set $L$ 
such that $K\Subset L\Subset\Omega$. 
We can also assume that 
$$i\partial u_k\wedge \dbar u_k\leq T_k$$
for some closed positive $(1,1)$-current $T_k$ such that $\|T_k\|\leq 1$.
Choose a smooth nonnegative function $\tau$ with compact support 
in $\Omega$ which is equal to 1 in a neighbourhood of $L$. It is enough to
prove by induction on $l$ that
$$\lim_{k \to \infty}\int_{\Omega} \tau  u_k \ddc v_{1,k} 
\wedge \cdots \wedge \ddc v_{l,k} \wedge \omega^{n-l}=0.$$
The case where $l=0$ is clear because $u_k \to 0$ in $L^2_{loc}$ 
by hypothesis, see the beginning of this section. 
Assume that the desired property holds for $l-1$ instead of $l$. We need to prove it for $l$.  

Define
$$R_k:= \ddc v_{2,k} \wedge \cdots \wedge \ddc v_{l,k} 
\wedge \omega^{n-l} \quad \text{and} \quad 
R:= \ddc v_{2} \wedge \cdots \wedge \ddc v_{l} \wedge \omega^{n-l}.$$
Let $v_{1,k,\varepsilon}$ and $v_{1, \varepsilon}$ 
be the standard regularizations of $v_{1,k}$ and $v_1$ for $\varepsilon>0$ 
a small constant (since $\tau$ has compact support, we can reduce $\Omega$ 
slightly in order to avoid problems near the boundary of $\Omega$). 
Define $v'_{1,k,\varepsilon}:= v_{1,k}- v_{1,k,\varepsilon}$. 
Observe that when $k$ tends to infinity, 
$v_{1,k,\varepsilon}$ decreases to $v_{1,\varepsilon}$ and hence 
$v'_{1,k,\varepsilon}$ tends to $v'_{1,\varepsilon}:=v_1- v_{1, \varepsilon}$ pointwise.
Write 
$$\int_{\Omega}\tau u_k \ddc v_{1,k} \wedge R_k= 
\int_{\Omega}\tau u_k \ddc v_{1,k,\varepsilon} \wedge R_k +
\int_{\Omega}\tau  u_k \ddc v'_{1,k, \varepsilon} \wedge R_k.$$
As $k$ tends to infinity, the first term in the last sum converges to 
$0$ by induction hypothesis. 
Denote the second term by $I_k(\varepsilon)$. It remains to check that 
$$\lim_{\varepsilon\to 0} \limsup_{k\to\infty} I_k(\varepsilon)=0.$$
By integration by parts, we get 
$$ I_k(\varepsilon)  =-\int_{\Omega}  u_k d \tau \wedge \dc v'_{1,k,\varepsilon}  
\wedge R_k - \int_{\Omega}\tau  d u_k  \wedge \dc v'_{1,k,\varepsilon} \wedge R_k.
$$
The first integral in the last line tends to 0 when $k$ tends to infinity 
because $u_k\to 0$ in $L^2_\loc(\Omega)$, 
$d\tau$ vanishes outside $L$ and 
$d \tau \wedge \dc v'_{1,k,\varepsilon}  \wedge R_k$ is 
bounded uniformly on $\Omega\setminus L$.
The second integral, denoted by $J_k(\varepsilon)$, 
satisfies the following estimates, thanks to the Cauchy-Schwarz inequality
\begin{eqnarray*}
|J_k(\varepsilon)|^2 &\le & \bigg(\int_\Omega \tau   d u_k \wedge \dc u_k  
\wedge R_k\bigg)  \bigg(\int_\Omega \tau  d v'_{1,k,\varepsilon} 
\wedge  \dc v'_{1,k,\varepsilon} \wedge R_k\bigg)\\
\nonumber
& \lesssim &\bigg(\int_\Omega \tau  T_k\wedge R_k \bigg)  
\bigg(\int_\Omega \tau   d v'_{1,k,\varepsilon} \wedge  \dc v'_{1,k,\varepsilon} 
\wedge R_k\bigg).
\end{eqnarray*}
The first factor in the last line is uniformly bounded, thanks to Lemma \ref{l:CLN+} 
and the fact that $\|T_k\| \leq1$.  
By integration by parts, the second factor is equal to
$$-\int_\Omega   v'_{1,k,\varepsilon} d\tau \wedge  \dc v'_{1,k,\varepsilon} 
\wedge R_k - \int_\Omega \tau    v'_{1,k,\varepsilon} \wedge  
\ddc v'_{1,k,\varepsilon} \wedge R_k.$$

Taking $k\to\infty$ and then $\varepsilon\to 0$, 
we see that the first term tends to 0 because $d\tau$ vanishes outside 
$L$ and $v'_{1,k,\varepsilon} d\tau \wedge  \dc v'_{1,k,\varepsilon} \wedge R_k$ 
is smooth and tends uniformly to 0 on $\Omega\setminus L$ 
thanks to properties of the convolution operator.
By continuity of wedge-product described at the beginning of this section, 
when $k$ tends to infinity, the second term tends to
$$ \int_\Omega \tau    v'_{1,\varepsilon} \wedge  \ddc v'_{1,\varepsilon} \wedge R.$$
Finally, when $\varepsilon$ decreases to 0, since $v_{1,\varepsilon}$ decreases to $v_1$, 
the last expression tends to 0, again, by using the continuity of the wedge-product.
This ends the proof of the lemma.
\endproof

We have the following result which will be extended later in Proposition \ref{p:W12-MA-general} to
every $u \in W^{1,2}_*$ which is not necessary continuous.
 
\begin{lemma}  \label{l:W12-L1-pot}
Let $\tau$ be a smooth function with  compact support in $\Omega$. 
Then there is a constant $c>0$ such that for every  
$u \in W^{1,2}_* \cap \cali{C}^0(\Omega)$ with $\|u\|_* \le 1$ 
and every p.s.h.\  functions $v_1, \ldots, v_{n-1}, w_1, w_2$ on 
$\Omega$ with values in $[0,1]$,  we have
$$\bigg| \int_\Omega \tau u \ddc (w_1- w_2)  \wedge 
\ddc v_1 \wedge \cdots \wedge \ddc v_{n-1} \bigg |  \le 
c \| w_1- w_2\|^{1/2}_{L^\infty(\Omega)}.$$
\end{lemma}

\proof 
Let $T$ be as in the proof of Lemma \ref{l:W12-d-dc-L2}. 
As in that lemma, by regularisation,  we can assume $u$ smooth.
Define $w:= w_1- w_2$ and 
$R:= \ddc v_1 \wedge \cdots \wedge \ddc v_{n-1}$.  
 Let $U\Subset \Omega$ be an open set containing 
 the support of $\tau$. By integration by parts, we get
$$\int_\Omega \tau u \ddc w  \wedge R= 
- \int_\Omega u d \tau \wedge \dc w  \wedge R-
\int_\Omega \tau du  \wedge \dc w  \wedge R.$$
Denote by $I_1, I_2$ the first and second term in the 
right-hand side of the last equality.   By Cauchy-Schwarz inequality, 
\eqref{e:d-dc-def} and  Lemma \ref{l:CLN+}, one has
\begin{align*}
|I_2| \lesssim \| du \wedge \dc u \wedge R\|_{U}^{1/2} \| d w \wedge 
\dc w \wedge R\|_{U}^{1/2} \lesssim  \| R \wedge T\|_U^{1/2} 
\|w\|^{1/2}_{L^\infty(\Omega)} \lesssim  \|w\|^{1/2}_{L^\infty(\Omega)}.  
\end{align*}
Similarly,
$$|I_1| \le  \| u^2 d \tau \wedge \dc \tau \wedge R\|_{U}^{1/2} \| d w 
\wedge \dc w \wedge R\|_{U}^{1/2} \lesssim  \|w\|^{1/2}_{L^\infty(\Omega)}$$
by Lemmas \ref{l:CLN+} and \ref{l:W12-d-dc-L2}. 
The result follows.
\endproof

In general,  the
potentials of a current $T$ satisfying \eqref{e:W12-Def} 
are not locally bounded. We will introduce below an operator which produces, 
from a bounded function $u \in W^{1,2}_*$, new bounded functions 
in $W^{1,2}_*$ such that their associated $(1,1)$-currents  have bounded  potentials. 
 
\begin{lemma} \label{l:W12-cutoff}  
Let $u \in W^{1,2}_*$ be such that $|u| \le 1$ and let $T$ be as in
\eqref{e:W12-Def}. Assume moreover that  $T= \ddc \varphi$ 
for some negative p.s.h.\  function $\varphi$ on $\Omega$. 
Define $\varphi_N:= \max\{\varphi, -N\}+N$ for a constant $N\geq 1$. 
Then the function $w:= \varphi_N u$ belongs to $W_*^{1,2}(\Omega')$ 
for every open set $\Omega'\Subset\Omega$ and it satisfies
$$i \partial w \wedge \overline \partial w \le 
2\pi N^2 \ddc (\varphi_N^2+ \varphi_{N+1}).$$ 
\end{lemma}
 
\proof  
Observe that $0 \le \varphi_N \le N$ and both $\varphi_N$ and 
$\varphi_N^2$ are p.s.h.\  So
the estimate in the lemma implies that $w$ belongs to $W^{1,2}_*(\Omega')$ 
and we only need to prove this estimate.

\medskip\noindent
{\bf Particular Case.} Consider first the case where $\varphi$ is continuous. Since
$\partial w = u \partial \varphi_N+ \varphi_N \partial u$, by Cauchy-Schwarz inequality, 
we have
$$i \partial w \wedge \overline \partial w \leq 
2 u^2 i\partial \varphi_N\wedge \dbar\varphi_N + 
2\varphi_N^2 i\partial u\wedge\dbar u
\leq 2\pi N^2(\ddc \varphi_N^2+ \ddc \varphi).$$
This implies the desired estimate because $\varphi=\varphi_{N+1}$ 
on the open set $\{\varphi>-N-1\}$ which contains the 
closed set $\{\varphi\geq -N\}$ and $w$ is supported by the last one.

\medskip\noindent
{\bf General Case.} Denote by $u_\varepsilon$ and $\varphi_\varepsilon$ the standard regularizations of $u$ and $\varphi$. We reduce slightly the domain $\Omega$ in order to avoid problems near the boundary. We have seen in the proof of  Lemma \ref{l:nice-CV} that $i\partial u_\varepsilon\wedge \dbar u_\varepsilon\leq \ddc\varphi_\varepsilon$. Define as above the functions $\varphi_{\varepsilon,N}$ and $w_\varepsilon$ associated to $\varphi_\varepsilon$ and 
$u_\varepsilon$. We obtain from the last case that 
$$i \partial w_\varepsilon \wedge \overline \partial w_\varepsilon \le 2\pi N^2 \ddc (\varphi_{\varepsilon,N}^2+ \varphi_{\varepsilon,N+1}).$$ 
When $\varepsilon$ decreases to 0, it is easy to see that $w_\varepsilon$ converges almost everywhere to $w$ and the right-hand side of the last inequality converges to $2\pi N^2 \ddc (\varphi_{N}^2+ \varphi_{N+1})$ because $\varphi_\varepsilon$ decreases to $\varphi$. 
The desired inequality in the lemma follows, see also the beginning of this section for the weak convergence in $W^{1,2}_*$.
\endproof

The following lemma gives a link  between the nice convergence and the convergence in capacity.
A related result was given in \cite{Vigny}.

\begin{lemma} \label{l:nice-cap-CV}
Let $(u_k)_k\in W^{1,2}_*\cap \cali{C}^0(\Omega)$ be a sequence converging nicely to $0$. Then $u_k$ converges to $0$ in capacity  as $k \to \infty$.  
\end{lemma}

\proof  Since the problem is local, we can assume that there are closed positive $(1,1)$-currents $T_k$ 
and negative p.s.h.\  functions $\varphi_k$ such that 
\begin{align} \label{ine-dinhnghiauktiennicely}
i \partial u_k \wedge \overline \partial u_k \le T_k =\ddc \varphi_k \quad \text{and} \quad \|T_k\| \leq 1.
\end{align}
Since $u_k \to 0$ nicely, we can also assume that  $\varphi_k$ decreases to some negative p.s.h.\   function $\varphi$ on $\Omega$.  
Let $K$ be a compact subset of $\Omega$ and $\delta>0$. We need to show that
$$\lim_{k\to\infty} \capK (\{|u_k|\geq \delta\}\cap K,\Omega)=0.$$

Consider a large positive constant $N$. 
By definition of capacity and Lemma \ref{l:W12-d-dc-L2}, we have
$$\capK(\{|u_k|\geq N\}\cap K,\Omega)\lesssim N^{-2}  \sup \Big\{ \int_K u_k^2 (\ddc v)^n, \  0 \le v \le 1  \text{ p.s.h.}\Big\} \lesssim N^{-2}.$$
On the other hand, by Lemma \ref{l:W12-operations}, $\min(|u_k|,N)$ converges to 0 nicely. Therefore, replacing $u_k$ with $\min(|u_k|,N)$ allows us to assume that the sequence $(u_k)_k$ is uniformly bounded.

Define $\varphi_{k,N}:=\max (\varphi_k,-N)+N$ and  $u_k':=\varphi_{k,N} u_k$. 
By Lemma \ref{l:W12-cutoff}, we have $i\partial u_k'\wedge \dbar u_k' \leq \ddc\varphi'_k$ with $\varphi'_k:= 2\pi N^2(\varphi_{k,N}^2+\varphi_{k,N+1})$. 
Observe that $\varphi_k'$ decreases to the bounded p.s.h.\  function $\varphi':= 2\pi N^2(\varphi_{N}^2+\varphi_{N+1})$ as $\varphi_k$ decreases to the p.s.h.\  function $\varphi$. Thus, $u'_k$ converges to $0$ nicely.  Since $\varphi_k \ge \varphi$, we have $\varphi_{k,N} \ge 1$ on $\{\varphi \ge  -N+1\}$. It follows that  $|u'_k| \ge |u_k|$ on $\{\varphi \ge -N+1\}$.   
On the other hand, using Lemma \ref{l:CLN+}, we obtain 
$$\capK\big( \{ \varphi < - N+1\}\cap K, \Omega \big) \lesssim N^{-1}.$$
It follows that
$$\capK \big( \{ |u_k| \ge \delta\}\cap K, \Omega \big) \lesssim \capK \big( \{ |u_k'| \ge \delta\}\cap K, \Omega \big)+N^{-1}.$$
Therefore, since the last estimate holds for every $N$, we only need to check that
$$\lim_{k\to\infty} \capK (\{|u'_k|\geq \delta\}\cap K,\Omega)=0.$$

By definition of capacity and Lemma \ref{l:W12-d-dc-L2}, we have 
\begin{eqnarray*}
\capK \big( \{ |u'_k| \ge \delta\}\cap K, \Omega \big) & \le & \delta^{-2} \sup \Big\{ \int_K u_k'^2 (\ddc v)^n, \   0 \le v \le 1  \text{ p.s.h.} \Big\} \\
& \lesssim & \bigg(\int_\Omega u_k'^2(\ddc\varphi'_k+\omega)^n\bigg)^{1/2^n}.
\end{eqnarray*} 
By Lemma \ref{l:W12-L2-limit-0}, the last integral tends to 0 as $k$ tends to infinity. The lemma follows.
\endproof

Here is the main result of this section which generalizes results by Vigny in \cite{Vigny}, see also Corollary \ref{c:nice-cap-limits} below.

\begin{theorem} \label{t:W12-rep-limit}
Let $u \in W^{1,2}_*$. Then there exists a Borel function $\tilde{u}$ defined everywhere on $\Omega$ except on a pluripolar set 
such that $\tilde{u}= u$ almost everywhere and   the following properties hold: 

$(i)$ For every open set $U \subset \Omega$ and every sequence  $(u_k)_k \subset W^{1,2}_*(U)\cap \cali{C}^0(U)$ such that $u_k \to u$ nicely in $W^{1,2}_*(U)$, we have $u_k \to \tilde{u}$ in capacity as $k \to \infty$. In particular, there exists a subsequence $(u_{j_k})_k$ of $(u_k)_k$ such that $u_{j_k}$ converges pointwise to $\tilde{u}$ everywhere on $U$ except on a pluripolar set. 

$(ii)$ For every constant $\varepsilon>0$, there exists an open subset $U$ of $\Omega$ with $\capK(U, \Omega) \le \varepsilon$ such that $\tilde{u}$ is continuous on $\Omega \backslash U$.

$(iii)$ If $\tilde{u}'$ is another Borel function satisfying $(i)$, then $\tilde{u}'= \tilde{u}$ on $\Omega$ except on a pluripolar set.  
 \end{theorem}
 
We see that $\tilde{u}$ is unique modulo pluripolar sets.  In analogy with the case of p.s.h.\  functions, we consider the above property (ii) as a quasi-continuity property of functions in $W^{1,2}_*$. 
 
\proof 
By Lemma \ref{l:Cauchy-limit}, (iii) is a direct consequence of (i). Moreover, (ii) follows from (i) by using exactly the arguments to prove the quasi-continuity of p.s.h.\  functions, see \cite[Th. 1.13]{Kolodziej05}. We now prove (i) and start with the construction of $\tilde u$.  

Let $u_\varepsilon$ be the standard regularization of $u$. Choose a sequence of positive numbers $(\varepsilon_k)_k$ decreasing to 0. 
We first prove the following claim.

\smallskip\noindent
{\bf Claim.} $(u_{\varepsilon_k})_k$ is a Cauchy sequence in $W^{1,2}_*(\Omega')$ with respect to capacity for every open set $\Omega'\Subset\Omega$.

\smallskip

Assume by contradiction that the claim is not true. 
By replacing $(\varepsilon_k)_k$ with a subsequence, we have the following property for some compact set $K$ and some open set $U$ with $K\Subset U\Subset\Omega$
$$\capK\big(\{|u_{\varepsilon_{2m}}-u_{\varepsilon_{2m+1}}|>\delta\}\cap K, U\big) \geq \kappa$$
for every $m=1,2,\ldots$, where $\delta$ and $\kappa$ are some positive numbers.

By Lemma \ref{l:nice-CV}, we have $u_{\varepsilon_k}\to u$ nicely. So we can write locally $i\partial u_{\varepsilon_k}\wedge \dbar u_{\varepsilon_k} \leq \ddc \psi_k$ for  some p.s.h.\  function $\psi_k$ which decreases to a p.s.h.\  function when $k$ tends to infinity. Observe that by Cauchy-Schwarz inequality
\begin{eqnarray*}
i \partial (u_{\varepsilon_{2m}}-u_{\varepsilon_{2m+1}}) \wedge \dbar  (u_{\varepsilon_{2m}}-u_{\varepsilon_{2m+1}})
&\leq& 2 i\partial u_{\varepsilon_{2m}} \wedge \dbar u_{\varepsilon_{2m}} + 2 i\partial u_{\varepsilon_{2m+1}} \wedge \dbar u_{\varepsilon_{2m+1}} \\
&\leq& \ddc (2\psi_{2m}+2\psi_{2m+1}).
\end{eqnarray*}
Therefore, $u_{\varepsilon_{2m}}-u_{\varepsilon_{2m+1}}$ tends to 0 nicely and hence in capacity, according to Lemma \ref{l:nice-cap-CV}.
This contradicts the above estimate on capacity for $u_{\varepsilon_{2m}}-u_{\varepsilon_{2m+1}}$ and completes the proof of the claim.

We apply Lemma \ref{l:Cauchy-limit} to the sequence $(u_{\varepsilon_k})_k$ and obtain a function $\tilde u$ equal almost everywhere to $u$ such that  $u_{\varepsilon_k}\to u$ in capacity. Now, the function $\tilde u$ is constructed. It remains to prove the first part of the assertion (i).
Since $u_k\to u$ nicely, we can write locally $i\partial u_k\wedge \dbar u_k\leq \ddc \varphi_k$ for some p.s.h.\  function $\varphi_k$ decreasing to a p.s.h.\  function when $k$ goes to infinity. Using that both $i\partial u_{\varepsilon_k}\wedge \dbar u_{\varepsilon_k}$ and $i\partial u_k\wedge \dbar u_k$ are bounded by $\ddc (\psi_k+\varphi_k)$, we see that 
the sequence
$$u_{\varepsilon_1},u_1,u_{\varepsilon_2},u_2,\ldots$$
converges to $u$ nicely. 
As above, we can show that this is a Cauchy sequence with respect to capacity.
Therefore, Lemma \ref{l:Cauchy-limit} implies that $u_k\to \tilde u$ in capacity.
\endproof

 From now on, by a \emph{good representative} of $u$, we always mean a function $\tilde{u}$ as in Theorem \ref{t:W12-rep-limit}. 
 It coincides with the representative constructed by Vigny in \cite{Vigny}. Theorem \ref{t:W12-rep-limit}(i) shows that the good representatives do not depend on the coordinates on $\Omega$. Therefore, this notion is well defined for functions on manifolds. If no confusion arises,  when refer to functions in $W^{1,2}_*$, \emph{we often use implicitly their good representatives}. The requirement that $u_k$ is continuous in Theorem \ref{t:W12-rep-limit}(i) is actually superfluous as shown in the following result. 

\begin{corollary}\label{c:nice-cap-limits}  
Let $u \in W^{1,2}_*$ and  $u_k\in W^{1,2}_*$ for $k \in \N$. Assume that  $u_k \to u$ nicely as $k \to \infty$.  Then $\tilde u_k \to \tilde{u}$ in capacity as $k \to \infty$, where $\tilde{u}, \tilde u_k$ are good representatives of $u$ and $u_k$ respectively (we often say that $u_k \to u$ in capacity for simplicity).  
\end{corollary}

\proof  
Note that the problem is local and we can always reduce the domain $\Omega$ in order to avoid problems near the boundary.
By hypothesis, we can write $i\partial u_k\wedge \dbar u_k \leq \ddc \varphi_k$ for some p.s.h.\  function $\varphi_k$ decreasing to a p.s.h.\  function when $k$ tends to infinity. 
 
We can apply Lemma \ref{l:nice-CV} for $u_k$ instead of $u$. Then we apply Theorem \ref{t:W12-rep-limit}(i) for the obtained sequence of functions. We deduce the existence of  $u'_k\in  W^{1,2}_* \cap \cali{C}^\infty(\Omega)$ such that 
$$\|u'_k - u_k \|_{L^2} \le 1/k, \quad \capK\big(\{|u'_k - \tilde u_k| \ge 1/k\} \cap K, \Omega\big) \le 1/k, \quad  \|u'_k\|_* \le c$$ 
for some constant $c$ independent of $k$ and $i \partial u_k' \wedge \overline \partial u_k' \le \ddc \varphi_k'$, where $\varphi'_k$ is a p.s.h.\  function. We can obtain $\varphi_k'$ from $\varphi_k$ using the standard regularization, see  the proof of Lemma \ref{l:nice-CV}. Since the sequence $(\varphi_k)_k$ decreases to a p.s.h.\  function, we can choose $u_k'$ (inductively on $k=1,2,\ldots$) so that 
$(\varphi_k')_k$ also decreases to some p.s.h.\  function. It follows that $u'_k \to u$ nicely in $W^{1,2}_*$. This allows us to apply Theorem \ref{t:W12-rep-limit}(i) again to infer that $u'_k \to \tilde{u}$ in capacity.  
Finally, the above capacity estimate (involving $u_k'-\tilde u_k$) implies the result. 
\endproof

We also need the following observation in order to work directly with good representatives. 

\begin{lemma} \label{l:operation-rep} 
Let $\tau: \R \to \R$ be a Lipschitz function. Let $u\in W^{1,2}_*$ and $\tilde{u}$ a good representative of $u$. Then, $\tau(\tilde{u})$ is a good representative of  $\tau(u) \in W^{1,2}_*$. In particular, the functions $\tilde{u}^+, \tilde{u}^-, |\tilde{u}|$ are good representatives of  $u^+, u^-, |u|$, and if $u_1, u_2 \in W^{1,2}_*$, then $\max\{\tilde{u}_1, \tilde{u}_2\} $ is a good representative of $\max\{u_1, u_2\}$, where $\tilde{u}_j$ is a good representative of $u_j$ for $j=1,2$. 
 \end{lemma}
 
 \proof  
Let $u_k:=u_{\varepsilon_k}$ be as in Lemma \ref{l:nice-CV}.
By Lemma \ref{l:W12-operations}(ii), we  have $\tau(u_k)\to \tau(u)$ nicely.  Thus, the result is a direct consequence of the second assertion of Theorem \ref{t:W12-rep-limit}(i). 
 \endproof

Let $v_1, \ldots, v_n$ be bounded p.s.h.\  functions on $\Omega$ and define $\mu:= \ddc v_1 \wedge \cdots \wedge \ddc v_n$. Note that $\mu$ is a positive measure  having no mass on  pluripolar sets because the capacity of every pluripolar set is zero. Theorem \ref{t:W12-rep-limit} allows us to integrate any nonnegative $u\in W^{1,2}_*$ against $\mu$ by putting $\langle \mu, u\rangle:= \langle \mu, \tilde{u} \rangle$. The definition is independent of the choice of  a good representative $\tilde{u}$ of $u$. More generally, we can defined in the same way $\langle \mu, \phi(u)\rangle$ for any positive Borel function $\phi$ defined everywhere on $\R.$

For every set $A \subset \Omega$ and every signed measure $\nu$, denote by $\|\nu\|_A$  the mass of $\nu$ on $A$.  The following properties will be useful in practice.

\begin{proposition}\label{p:W12-MA-general}  
The estimate \eqref{l:W12-L2-MA-1} and Lemma \ref{l:W12-L1-pot} hold for all functions 
$u \in W^{1,2}_*$ with $\|u\|_* \le 1$ which are not necessarily continuous. Moreover, if $u_k \to u$ in $W^{1,2}_*$ nicely, then 
$$\lim_{k\to\infty}\|(u_k- u) \mu\|_K =0$$
for every compact set $K\subset \Omega$. 
\end{proposition}
 
\proof  
We first prove  \eqref{l:W12-L2-MA-1} for every $u \in W^{1,2}_*$ with $\|u\|_* \le 1$.
By Lemmas \ref{l:W12-operations} and \ref{l:operation-rep}, without loss of generality, we can suppose that $u$ is a bounded function. 
The point here is that the constants involving in our estimates do not depend on $u$.
By Lemma \ref{l:nice-CV}, we can find a sequence of smooth $u_k \in W^{1,2}_*$ (shrinking $\Omega$ if necessary) so that  
$$\|u_k\|_{L^\infty} \le \|u\|_{L^\infty}, \quad \|u_k\|_* \le \| u\|_*$$
 and $u_k \to u$ nicely. By Theorem \ref{t:W12-rep-limit} and extracting a subsequence if necessary, we can assume that $u_k \to u$ pointwise except on a pluripolar set. This together with Lebesgue's dominated convergence theorem gives
 $$\int_K |u|^2 d \mu  = \lim_{k \to \infty} \int_K |u_k|^2 d \mu.$$
The last integral is bounded uniformly by a constant times $\|u_k\|_*^2$ according to Lemma \ref{l:W12-d-dc-L2}. 
Hence,  \eqref{l:W12-L2-MA-1} holds for every $u$.

Observe that Lemma \ref{l:W12-L1-pot} for general $u$ can be obtained using the above functions $u_k$ and the last assertion in the proposition. Therefore, it remains to prove  this assertion.
Since $\mu\leq (\ddc (v_1+\cdots+v_n))^n$, we can replace all $v_j$ by $v_1+\cdots + v_n$ and assume that $\mu$ is a Monge-Amp\`ere measure with bounded potential.
Using Lemmas \ref{l:W12-operations} and \ref{l:operation-rep}, we can assume that $u_k$ and $u$ are  nonnegative.  Let $N$ be a big constant. 
Define 
$$u_{k,N}:=\min \{u_k, N\} \quad \text{and} \quad u_N:=\min\{u, N\}.$$
We have 
$0\leq u_{k,N} \le N$, $u_{k,N}=  u_k$ on $\{u_k \ge N\}$ and  similar properties for $u$ in place of  $u_k$. Observe that $u_{k,N} \to u_N$ nicely as $k \to \infty$, see Lemma \ref{l:W12-operations}(ii). Hence, by Corollary \ref{c:nice-cap-limits}, 
we have $u_{k,N} \to u_N$ in capacity.  

Using the first assertion in the proposition, we have 
$$\|  (u_k - u_{k,N}) \mu \|_K \le \int_{K \cap \{u_k \ge N\}} u_k d \mu \le  N^{-1} \int_K |u_k|^2 d \mu \lesssim 1/N$$
 and a similar estimate for $(u-u_N)\mu$. Together with the equality 
$$u_k-u = (u_k - u_{k,N})+ (u_{k,N} - u_N)  + (u_N - u),$$
we infer 
 $$\|(u_k- u)\mu\|_K  \le  \| (u_{k,N} - u_N) \mu\|_K+ O(1/N).$$
Denote by  $I_{k}$ the left-hand side of the last inequality. 
Let $\delta>0$ be a small constant. 
Using that $\mu$ is a Monge-Amp\`ere measure with bounded potential and the inequalities $0\leq u_{k,N}, u_N \leq N$,  we deduce from the last estimate that
\begin{eqnarray*} 
I_{k} & \le &  \bigg( \int_{K \cap \{|u_{k,N}- u_N|\ge \delta\}} |u_{k,N} - u_N| d\mu \bigg) + \bigg(\int_{K \cap \{|u_{k,N}- u_N| <\delta\}} |u_{k,N} - u_N| d\mu \bigg) +c/N \\
& \le & c \Big[N \capK\big( K \cap \{|u_{k,N}- u_N|\ge \delta\}, \Omega\big)+ \delta+ 1/N\Big]
\end{eqnarray*}
for some constant $c$ independent of $k,N,\delta$.  Letting $k$ tend to infinity, since  $u_{k,N} \to u_N$ in capacity, we obtain 
$$\limsup_{k \to \infty}I_{k}\le c[\delta+ 1/N]$$
for every $\delta, N>0$. Thus, $\lim_{k\to \infty}I_{k} =0$ and the proof is complete.
\endproof


\section{Proof of the main results}

We will consider Corollary \ref{cor-Lp-CV} at the end of this section. 
The proof of Theorem \ref{t:main} consists of two main steps. In Step 1, we show how to reduce the question to the case where $v_1, \ldots, v_n$ are smooth. In Step 2, we prove the desired result in the latter case.  Here is the precise formulation for Step 1.

\begin{proposition} \label{pro-step2tusmoothtoiMA} If Theorem \ref{t:main} holds for $v_j(x)= \|x\|^2$ for every $1 \le j \le n$,  then it holds (possibly with different constants $\alpha$ and $C$) for every H\"older continuous p.s.h.\  function $v_1, \ldots, v_n$ with H\"older exponent $\beta \in (0,1]$ on $\Omega$ such that $\|v_j\|_{\cali{C}^\beta} \le 1$ for $1 \le j \le n$.   
\end{proposition}

For every H\"older continuous function $v$ on $\Omega$, recall that the standard H\"older norm $\|v\|_{\cali{C}^\beta}$ is given by   
$$\|v\|_{\cali{C}^\beta}:= \sup_{x,y \in \Omega, x \not = y} \frac{|v(x)- v(y)|}{|x-y|^\beta}.$$

\proof  Without loss of generality, we can assume $u \ge 0$, see Lemma \ref{l:W12-operations}. Let $K\Subset \Omega$ be a compact set.  By hypothesis,  there exist  strictly positive constants $\alpha$ and $c$ such that for every $u \in W^{1,2}_*(\Omega)$ with $\|u\|_* \le 1,$ we have
\begin{align} \label{ine-MoserTrLeb}
\int_K e^{\alpha u^2} d \Leb \le c.
\end{align} 
Let $\omega$ be the standard K\"ahler form on $\C^n$. Let $l$ be an integer in $[0,n].$   Put $u_N:= \min \{u,N\}$ which is in $W^{1,2}_*$ with bounded $*$-norm. 

\medskip\noindent
\textbf{Claim.} There exist positive constants $\alpha$ and $c$ such that  for every constant $N>0,$ we have
$$\int_{K} (u-u_N) \ddc v_1 \wedge \cdots \wedge \ddc v_l \wedge \omega^{n-l} \le c e^{-\alpha N^2},$$
uniformly in  p.s.h.\  functions $v_1, \ldots, v_l$ on $\Omega$ and $u \in W^{1,2}_*(\Omega)$ such that $\|v_j\|_{\cali{C}^\beta} \le 1$ for every $1 \le j \le l$ and $\|u\|_* \le 1$. 

\medskip \noindent
Note that since $u-u_N \ge 1$ on $\{u \ge N+1\}$, the claim with $l=n$ implies the following inequality
 $$\int_{\{u\ge N+1\}\cap K} \ddc v_1 \wedge \cdots  \wedge \ddc v_n \le c e^{-\alpha N^2}.$$
From this estimate, we easily deduce the desired assertion \eqref{e:main-th} (we change the constants $\alpha$ and $c$ if necessary).

It remains to prove the claim and this will be done by induction on $l$. When $l= 0,$ the claim is a direct consequence of  \eqref{ine-MoserTrLeb} (again, we change the constants $c$ and $\alpha$ if necessary).  Assume that the claim holds for $l-1$ instead of $l$. We need to prove it for $l$. 
Choose a nonnegative smooth function $\tau$ supported by a compact set $K'\Subset \Omega$ such that $\tau=1$ on $K$. Since $u-u_N\geq 0$, we only need to bound the integral
$$I:= \int_\Omega \tau (u- u_N) \ddc v_1 \wedge \cdots \wedge \ddc v_{l} \wedge \omega^{n-l}.$$

Let $v_{1,\varepsilon}$ be the standard regularization of $v_1$ for $0<\varepsilon<1$, see the beginning of Section \ref{s:Sobolev}. As $v_{1,\varepsilon}$ is obtained from $v_1$ by convolution and $\|v_1\|_{\Cc^\beta}\leq 1$, we have $\|v_{1,\varepsilon}-v_1\|_{L^\infty}\lesssim \varepsilon^\beta$  and $\|v_{1,\varepsilon}\|_{\cali{C}^2} \lesssim \varepsilon^{-2}$.  By induction hypothesis applied to $K'$ instead of $K$, one gets
$$I_1:=\int_\Omega \tau  (u-u_N)  \ddc v_{1, \varepsilon} \wedge \ddc v_2 \wedge \cdots \wedge \ddc v_{l} \wedge \omega^{n-l} \lesssim  \varepsilon^{-2} e^{-\alpha N^2}$$
for some constant $\alpha>0$. Define
$$I_2:=\int_\Omega \tau  (u- u_{N}) \ddc (v_1- v_{1,\varepsilon}) \wedge \ddc v_2 \wedge  \cdots \wedge \ddc v_{l} \wedge \omega^{n-l}.$$
By Lemma \ref{l:W12-L1-pot} and Proposition \ref{p:W12-MA-general}, we have
$$I_2 \lesssim \|v_1 - v_{1, \varepsilon}\|^{1/2}_{L^\infty} \lesssim \varepsilon^{\beta/2}.$$
Since $I = I_1 + I_2$, we deduce that
$$I \lesssim  \varepsilon^{-2} e^{-\alpha N^2}+ \varepsilon^{\beta/2}.$$
Letting $\varepsilon:= e^{-2(4+\beta)^{-1}\alpha N^2}$ gives $I \lesssim e^{-\beta(4+\beta)^{-1}\alpha N^2}.$ We obtain the desired claim by changing $\alpha$ to $\beta(4+\beta)^{-1}\alpha$.  This ends the proof of the proposition.
\endproof

It remains to prove Theorem \ref{t:main} for $v_j=\|x\|^2$. In this case, $\mu:= \ddc v_1 \wedge \cdots \wedge \ddc v_n$ is the standard  volume form on $\Omega$.  The idea is to use suitable slicing in order to reduce the problem to the case of dimension 1.  We first recall some facts about the slicing theory of closed positive currents. We refer to \cite{Federer,DS_superpotential} for details.  Our setting is simpler because  we only work with $(1,1)$-currents. 

\medskip 

\noindent
{\bf Slicing theory.}  Let $U$ and $V$ be bounded open subsets of $\C^{m_1}$ and $\C^{m_2}$ respectively. Let $\pi_U: U \times V \to U$ and $\pi_V: U \times V \to V$ be the natural projections. Observe that  if $R$ is a form with $L^1_{loc}$ coefficients (which is not necessarily closed or positive), we can always define the restriction $R_z$ of $R$ to the fiber $\pi_V^{-1}(z)$ for almost every  $z\in V$ (with respect to  the Lebesgue measure on $V$).  

Consider now a closed positive $(1,1)$-current $R$ on $U \times V$.  Write $R= \ddc w$ locally, where $w$ is a p.s.h.\  function. For $z\in V$, we define the slice $R_z$ of  $R$ on $\pi_V^{-1}(z)$ to be $\ddc \big(w(\cdot, z)\big)$ which is a closed positive $(1,1)$-current on $\pi_V^{-1}(z)$.  Let $A$ be the set of $z$ so that  $w(\cdot, z) \equiv -\infty$. Observe that $A$ is pluripolar and for $z\not \in A$, the slice $R_z$ is well-defined. One can see that the definition  of the slice $R_z$ is independent of the choice of a local potential $w$ of $R$. 
 
Let $\chi_{m_2}$ be a nonnegative smooth radial function with compact support on $\C^{m_2}$ such that $\int_{\C^{m_2}} \chi d \Leb =1$ and for every constant $\varepsilon >0$, we put $\chi_{m_2,\varepsilon}(z):= \varepsilon^{-2 m_2} \chi_{m_2}(\varepsilon^{-1}z)$.
The following result is straightforward. We just notice that  (iii) is a direct consequence of (ii). 

\begin{lemma} \label{le-slicetheory} 
$(i)$ Let $\Phi$ be a smooth form of suitable bi-degree with compact support in $U\times V$. Let $\Theta(z)$ be a smooth volume form on  $V$.  Then we have 
$$\langle R, \Phi \wedge \Theta(z) \rangle = \int_{z\in Z} \langle R_z, \Phi \rangle \Theta(z).$$

 $(ii)$ Then for $z_0 \not \in A$, we have 
$$\lim_{\varepsilon \to 0} R \wedge \pi_V^*\big(\chi_\varepsilon(z-z_0) \Leb(z)\big)= R_{z_0},$$
 where we identified $R_{z_0}$ with a current on $U \times V$  (when   $R$ is a $(1,1)$-form with $L^1_{loc}$ coefficients which is not necessarily closed or positive, then the same conclusion holds for almost every $z \in V$).  

$(iii)$ Let $R'$ be another closed positive $(1,1)$-current on $U \times V$ or a real $(1,1)$-form with $L^1_{loc}$ coefficients. Assume that $R' \le R$ on $U \times V$. Then for almost every $z \in Z$, we have $R'_z \le R_z$. 
\end{lemma}

\medskip
We continue the proof of Theorem \ref{t:main} for $v_j= \|x\|^2$. 
We will need the following lemma.

\begin{lemma} \label{l:Sobolev-mean-T} 
Let $\eta(x,z)$ be a $(1,0)$-form with $L^2$ coefficients on $U \times V$. Let $T$ be a closed positive $(1,1)$-current of mass at most $1$ on $U \times V$ such that $i \,\eta \wedge \overline{\eta} \le T$ on $U \times V$. Define $\tilde{\eta}(x):= \int_{z\in V} \eta(x,z) d \Leb(z)$. Then there exists a positive constant $C$ independent of $\eta$ and $T$ such that 
$$i \, \tilde{\eta} \wedge \overline{\tilde{\eta}} \le C \, (\pi_U)_*\big(T \wedge \Leb(z)\big),$$
where  $\Leb(z)$ denotes both the Lebesgue measure and the standard volume form on $V$. Moreover, $(\pi_U)_*\big(T \wedge \Leb(z)\big)$ is a closed positive $(1,1)$-current on $U$ whose mass is bounded by a constant independent of $\eta$ and $T$.
\end{lemma}

\proof  Let $\Phi$ be a weakly positive smooth form of right bi-degree 
with compact support on $U$.  Using  the Cauchy-Schwarz inequality implies that 
\begin{align} \label{ine-sumDiracetattabar}
\langle i \, \tilde{\eta} \wedge \overline{\tilde{\eta}}, \Phi \rangle &= \int_{(z,z') \in V^2}  \langle i \,\eta_{z} \wedge \overline{\eta}_{z'}, \Phi \rangle  d \Leb(z,z') \\
\nonumber
& \le \frac{1}{2}\int_{(z,z') \in V^2}  \big(\langle i \,\eta_{z} \wedge \overline{\eta}_{z}, \Phi \rangle + \langle i \,\eta_{z'} \wedge \overline{\eta}_{z'}, \Phi \rangle\big)  d \Leb(z,z') \\
\nonumber
& \lesssim \int_{V} \langle i \,\eta_{z} \wedge \overline{\eta}_{z}, \Phi \rangle d \Leb(z) \\
\nonumber 
& =  \int_{U \times V}  i \,\eta \wedge \overline{\eta}\wedge \Leb(z) \wedge \pi_V^*(\Phi)  \le  \big \langle (\pi_U)_*\big(T \wedge \Leb(z)\big), \Phi \big \rangle. 
\end{align}
This implies the first assertion in the lemma.

If $\omega(x)$ denotes the standard K\"ahler form on $U$, then the mass of $(\pi_U)_*\big(T \wedge \Leb(z)\big)$ is equal to the mass of the measure
$$T \wedge \Leb(z)\wedge \omega(x)^{m_1-1}.$$ 
Clearly, this mass is bounded by a constant because the mass of $T$ is at most equal to 1 by assumption. It remains to show that $(\pi_U)_*\big(T \wedge \Leb(z)\big)$ is closed. Let $0\leq\chi_k(z)\leq 1$ be a sequence of smooth functions with compact support in $V$ which increases to 1. We have
$$(\pi_U)_*\big(T \wedge \Leb(z)\big) = \lim_{k\to\infty} (\pi_U)_*\big(T \wedge \chi_k(z) \Leb(z)\big).$$
Observe that $\chi_k(z) \Leb(z)$ is closed because it is of maximal degree in $z$. Therefore, the current $T \wedge \chi_k(z) \Leb(z)$ is also closed.
Since $\pi_U$ is proper on the support of $T \wedge \chi_k(z) \Leb(z)$, we deduce that $ (\pi_U)_*\big(T \wedge \chi_k(z) \Leb(z)\big)$ is closed and the last identity implies the result.
\endproof

\begin{lemma} \label{le-chondiemtam2}  Let $U,V$ be open subsets in $\C^{m_1}, \C^{m_2}$ respectively.  Let $u$ be a locally integrable function in $U \times V$ such that $\partial u \in L^2_{loc}(U \times V)$.  Let $T$ be   a closed positive $(1,1)$-current  on $U \times V$ such that $i \partial u \wedge \overline \partial u \le T.$   Then, for almost every $z \in V$, we have that $\partial (u |_{U \times \{z\}}) \in L^2_{loc}(U)$ and  
\begin{align}\label{ine-sliceuPsix'W12}
i \partial (u |_{U \times \{z\}}) \wedge \overline \partial (u |_{U \times \{z\}}) \le T |_{U \times \{z\}}.
\end{align}
\end{lemma}

\proof By Fubini's theorem, for almost everywhere $z$, the forms $\partial (u |_{U \times \{z\}})$ and $\partial (u |_{U \times \{z\}}) \wedge \overline \partial (u |_{U \times \{z\}})$ are equal to the slice of $\partial u$, $\partial u \wedge \overline \partial u$ along $U \times \{z\}$, respectively. This combined with  Lemma \ref{le-slicetheory} gives the desired assertion.  
\endproof

Let $\D$ be the unit disk in $\C$. We will need the following basic observation which can be deduced using the Riesz representation of subharmonic functions (see \cite[Theorem 3.3.6]{Hormander}). 

\begin{lemma} \label{le-chanmasssubharwmonic} Let $\varphi$ be a negative subharmonic function on $\D$ and $\varphi(0) \ge -1$. Let $K$ be a compact subset of $\D$. Then there exists a constant $C$ independent of $\varphi$ such that $\|i \partial \overline \partial \varphi\|_K \le C$. 
\end{lemma}

\begin{proof}[End of the proof of Theorem \ref{t:main}] We prove (\ref{e:main-th}) by induction. For $n=1$, this is Theorem \ref{t:Moser}. We assume that (\ref{e:main-th})  holds for every dimension at most $n-1$. We need to prove that (\ref{e:main-th})  holds for dimension $n$.  Let $u \in W^{1,2}_*$ with $\|u\|_*=1.$  Let $T$ be a closed positive current of bi-degree $(1,1)$ such that 
\begin{align}\label{ine-uTfinalstep}
i \partial u \wedge \overline \partial u \le T
\end{align}
 and $\|T\|_\Omega \le 1.$ Let $K \Subset  \Omega$.   Since our problem is local, by solving $\ddc$-equation, without loss of generality, we can assume that $T= \ddc \varphi$ for some p.s.h.\  function $\varphi$ on $\Omega$ and $\|\varphi\|_{L^1(\Omega)}\le C$, where $C$ is a constant independent of $T$ (see \cite[Lemma 2.1]{Vu_nonkahler_topo_degree} for example). Thus,  for every constant $M>0$ and  $F_{M}: =\{|\varphi|\le M\}$, we get 
\begin{align}\label{ine-omegabaclslashAm}
\Leb(\Omega \backslash F_M)= \int_{\{|\varphi| > M\}} \omega^n \le M^{-1} \int_{\{|\varphi|> M\}} |\varphi| \omega^n \lesssim M^{-1}.
\end{align}
By decomposing $K$ into the union of a finite number of small compact sets, it suffices to consider the case where  the diameter of $K$ is as small as we want. Hence, by using a change of coordinates,  we can assume that the closure of the unit ball $\B$ is contained in $\Omega$ and $K$ is contained in $\{3/4 \le \|x\| \le 4/5\}$. By (\ref{ine-omegabaclslashAm}), we see that if $M$ is big enough, there is a point $a \in \B$ with $\|a\| \le 1/100$ so that $\varphi(a)> -M$. Using a linear change of coordinates again, we can assume furthermore that $a=0$. Thus, $\varphi(0)>-M$ for some fixed constant $M$ big enough and independent of $u,T.$

The set of complex lines passing through the origin is parameterized by the complex projective space $\P^{n-1}$. For $y \in \P^{n-1}$, denote by $L_y$ the complex line given by $y$ and $\D_y:= \B \cap L_y$  which  is the unit disc of $L_y$. Put $A_y:= \{x \in \B: \|x\| > 1/10 \} \cap \D_y$. We can identify $A_y$ with an annulus $A:= \{z \in \D: |z|> 1/10\}$ in $\D$ (here we consider a linear isometry from $\D_y$ to the unit disk $\D$). 
 Let   $u_y:= u|_{L_y}$ for every $y \in \P^{n-1}$ and 
$$ \underline{u}(y):= \int_{z \in A_y} u_y(z)  d \Leb(z), \quad \tilde{u}_y:= u_y - \underline{u}(y).$$
The last functions are well-defined for almost every $y$.  Denote by $T|_{\D_y}$ the slice of $T$ along $\D_y$.  Recall that $T|_{\D_y}= \ddc (\varphi|_{\D_y})$. Observe that since $\varphi|_{\D_y}(0)= \varphi(0)>-M$, the mass of $T|_{\D_y}$ on $\D_y$ is bounded by  a constant independent of $y,T$ (note here that $\overline \B \subset \Omega$). On the other hand,  by Lemma  \ref {le-chondiemtam2},  for almost every $y$, we have that  $u_y \in W^{1,2}_*(\D_y)$ and using the definition of $\tilde{u}_y$,  
$$\|\tilde u_y\|^2_* \lesssim \|\tilde u_y\|^2_{L^2}+  \|  T|_{\D_y}\| \lesssim  \|  T|_{\D_y}\| $$
by Poincar\'e's inequality (\cite[page 275]{Evans_pde}).  It follows that 
\begin{align}\label{ine0chuansaocuauPsix0}
\|\tilde u_y\|_* \lesssim \|  T|_{\D_y}\| \lesssim 1.
\end{align}
Let $\alpha>0$ be a fixed small constant. Let $A':= \{z \in \D:  2/3 \le |z| \le 5/6\}$ which is  compact in $A$. Let $A'_y$ be the image of $A'$ under the natural identification $\D_y \approx \D$.   By construction, we have $K \subset \cup_{y \in \P^{n-1}} A'_y$. Note that $|u_y|^2 \le 2(|\tilde{u}_y|^2+ |\underline{u}(y)|^2)$.  By this and  Fubini's theorem, we get 
\begin{align*}
\int_{K} e^{\alpha |u|^2} d \Leb & \lesssim  \int_{y \in \P^{n-1}} \bigg(\int_{A'_y} e^{\alpha | u_y|^2} d \Leb\bigg) d \Leb(y)\\
& \lesssim  \int_{y \in \P^{n-1}} \bigg(\int_{A'_y} e^{2 \alpha (| \tilde{u}_y|^2+ |\underline{u}(y)|^2)} d \Leb\bigg) d \Leb(y)\\
& \lesssim  \int_{y \in \P^{n-1}} \bigg(\int_{A'_y} e^{4 \alpha | \tilde u_y|^2} d \Leb\bigg) d \Leb(y)+ \int_{y \in \P^{n-1}} e^{4\alpha |\underline{u}(y)|^2} d \Leb(y)
\end{align*}
by H\"older's inequality. Denote by $I_1, I_2$ the first and second terms respectively in the right-hand side of the last inequality. By (\ref{ine0chuansaocuauPsix0}) and induction hypothesis, the number $I_1$ is uniformly bounded in $u$ for some constant $\alpha>0$ independent of $u$ (note that $A'_y \approx A' \Subset A \approx A_y$). It remains to check that $I_2$ is also uniformly bounded.    

Cover $\P^{n-1}$ by a finite number of small local charts $U$ such that $U \times A$ is identified with a  chart in $\B$ and $\{y\} \times A$  is identified with $A_y$ for every $y \in U$.  Thus, we get  
$$\|\underline{u}\|_{L^2(U)} \lesssim \| u \|_{L^2(\B)} \lesssim 1.$$
By (\ref{ine-uTfinalstep}) we can apply    Lemma \ref{l:Sobolev-mean-T} to $\eta(y,z):= u(y,z)$ on $U \times A$. Hence one obtains
$$i \partial \underline{u} \wedge \overline \partial \underline{u} \lesssim \pi_*\big(T \wedge (i dz \wedge d \bar z)\big),$$
where  $\pi$ denotes  the natural projection from $U \times  A$ to the first component. Moreover, the right-hand side of the last inequality is a closed positive $(1,1)$-current on $U$ with bounded mass.  We deduce that $\underline{u} \in W^{1,2}_*(U)$ with $*$-norm uniformly bounded.  By induction hypothesis applied to $\underline{u}$, the integral $I_2$ is bounded uniformly for some constant $\alpha>0$ (we can slightly reduce $U$ in order to apply the induction hypothesis).  This finishes the proof.  
\end{proof}

We have the following result.

\begin{proposition} \label{p:CV-exp}
Let $K$ be a compact subset of an open set $\Omega$ in $\C^n$. Let $(u_k)_k\subset W^{1,2}_*(\Omega)$ be a sequence converging weakly to a function $u\in W^{1,2}_*(\Omega)$. Assume that $\|u_k\|_*\leq 1$. Then there is a positive constant $\alpha$ depending only on $K$ and $\Omega$ such that 
$$\lim_{k\to\infty} \|e^{\alpha (u_k-u)^2}-1\|_{L^1(K)}=0.$$  
\end{proposition}
\proof
Observe that $\|u\|_*\leq 1$ and hence $\|u_k-u\|_*\leq 2$. 
Define $v_k:=u_k-u$, and for $N\in\R_{>0}$, define $K_{k,N}:=\{|v_k|\geq N\}\cap K$. Note that $\|v_k\|_*$ is bounded uniformly in $k$.
By this and Theorem \ref{t:main}, there are positive constants $\alpha$ and $c$ such that (we change the constant $\alpha$ in order to get the factor 2)
$$\int_K e^{2\alpha |v_k|^2} d\Leb \leq c.$$
Using that $|v_k|\geq N$ on $K_{k,N}$, we deduce that
$$\Leb(K_{k,N}) \leq e^{-2\alpha N^2} \int_{K_{k,N}} e^{2\alpha |v_k|^2} d\Leb \leq c e^{-2\alpha N^2}$$
and by Cauchy-Schwarz inequality
\begin{equation} \label{e:K_k_N}
 \int_{K_{k,N}} e^{\alpha |v_k|^2} d\Leb \leq \Leb(K_{k,N})^{1/2} \Big( \int_{K_{k,N}} e^{2\alpha |v_k|^2} d\Leb\Big)^{1/2} \leq c e^{-\alpha N^2}.
 \end{equation}

On another hand, on $K\setminus K_{k,N}$ with $N$ fixed, we have $|v_k|\leq N$ and hence $e^{\alpha v_k^2}-1 \lesssim v_k^2$. As mentioned at the beginning of Section \ref{s:Sobolev}, Rellich's theorem implies that $\|v_k\|_{L^2(K)}\to 0$. We deduce that 
$$\lim_{k\to\infty} \|e^{\alpha v_k^2}-1\|_{L^1(K\setminus K_{k,N})}=0.$$
This, together with \eqref{e:K_k_N}, imply that
$$\limsup_{k\to\infty} \|e^{\alpha v_k^2}-1\|_{L^1(K)}\leq c e^{-\alpha N^2}.$$
Since this estimate holds for every $N$, the proposition follows.
\endproof

\proof[Proof of Corollary \ref{cor-Lp-CV}.]
By H\"older's inequality, we can assume that $p\geq 2$. 
Then the corollary is a direct consequence of Proposition 
\ref{p:CV-exp} because $|t|^p\lesssim e^{\alpha t^2}-1$. 
\endproof

\begin{example} \label{ex-contiu} We consider the 1-dimensional case. Let $\D$ be  the unit disc in  $\C$.  Let  $u:=(1-\log|z|)^{1/3}$ for $z \in \D$. One can check directly that $u\in W^{1,2}_*(\D)$. Consider $f(x):= x^{-1} (-\log |x|)^{-3}$ for  $x \in (0,1/2]$,  and the positive measure $\mu:=\bold{1}_{[0,1/2]} f(x) d x$.  Observe that $\log|z|$ is in $L^{3/2}(\mu)$ but not in $L^3(\mu)$. We have $\mu = \ddc v$ for $v(w):= \int_\D \log |z-w| d \mu(w)$. Using the concrete form of $f$, one can prove that $v$ is continuous on $\D$, and  $e^{\alpha u^2}$ is not locally integrable with respect to $\mu$ for any $\alpha>0$.
\end{example}

\bibliography{biblio_family_MA,biblio_Viet_papers}

\begin{thebibliography}{10}

\bibitem{Bedford_Taylor_82}
{\sc E.~Bedford and B.~A. Taylor}, {\em A new capacity for plurisubharmonic
  functions}, Acta Math., 149 (1982), pp.~1--40.

\bibitem{ChangYang}
{\sc S.-Y.~A. Chang and P.~C. Yang}, {\em Conformal deformation of metrics on
  {$S^2$}}, J. Differential Geom., 27 (1988), pp.~259--296.

\bibitem{Chern_Levine_Nirenberg}
{\sc S.~S. Chern, H.~I. Levine, and L.~Nirenberg}, {\em Intrinsic norms on a
  complex manifold}, in Global {A}nalysis ({P}apers in {H}onor of {K}.
  {K}odaira), Univ. Tokyo Press, Tokyo, 1969, pp.~119--139.

\bibitem{Cianchi}
{\sc A.~Cianchi}, {\em Moser-{T}rudinger inequalities without boundary
  conditions and isoperimetric problems}, Indiana Univ. Math. J., 54 (2005),
  pp.~669--705.

\bibitem{Demailly_ag}
{\sc J.-P. Demailly}, {\em Complex analytic and differential geometry}.
\newblock \url{https://www-fourier.ujf-grenoble.fr/~demailly}.

\bibitem{DiNezzaGuedjLu-MT}
{\sc E.~{Di Nezza}, V.~{Guedj}, and C.~H. {Lu}}, {\em {Finite entropy vs finite
  energy}}, {Comment. Math. Helv.}, 96 (2021), pp.~389--419.

\bibitem{DLW}
{\sc T.-C. Dinh, L.~Kaufmann, and H.~Wu}, {\em Dynamics of holomorphic
  correspondences on {R}iemann surfaces}.
\newblock \url{https://arxiv.org/abs/1808.10130}, 2018.
\newblock to appear in Internat. J. Math.

\bibitem{DinhVietanhMongeampere}
{\sc T.-C. Dinh and V.-A. Nguy{\^e}n}, {\em Characterization of
  {M}onge-{A}mp\`ere measures with {H}\"older continuous potentials}, J. Funct.
  Anal., 266 (2014), pp.~67--84.

\bibitem{DVS_exponential}
{\sc T.-C. Dinh, V.-A. Nguy{\^e}n, and N.~Sibony}, {\em Exponential estimates
  for plurisubharmonic functions and stochastic dynamics}, J. Differential
  Geom., 84 (2010), pp.~465--488.

\bibitem{DS_decay}
{\sc T.-C. Dinh and N.~Sibony}, {\em Decay of correlations and the central
  limit theorem for meromorphic maps}, Comm. Pure Appl. Math., 59 (2006),
  pp.~754--768.

\bibitem{DS_superpotential}
\leavevmode\vrule height 2pt depth -1.6pt width 23pt, {\em Super-potentials for
  currents on compact {K}\"ahler manifolds and dynamics of automorphisms}, J.
  Algebraic Geom., 19 (2010), pp.~473--529.

\bibitem{Evans_pde}
{\sc L.~C. Evans}, {\em Partial differential equations}, vol.~19 of Graduate
  Studies in Mathematics, American Mathematical Society, Providence, RI,
  second~ed., 2010.

\bibitem{Federer}
{\sc H.~Federer}, {\em Geometric measure theory}, Die Grundlehren der
  mathematischen Wissenschaften, Band 153, Springer-Verlag New York Inc., New
  York, 1969.

\bibitem{Hiep_holder}
{\sc P.~H. Hiep}, {\em H\"older continuity of solutions to the
  {M}onge-{A}mp\`ere equations on compact {K}\"ahler manifolds}, Ann. Inst.
  Fourier (Grenoble), 60 (2010), pp.~1857--1869.

\bibitem{Hormander}
{\sc L.~H{\"o}rmander}, {\em Notions of convexity}, vol.~127 of Progress in
  Mathematics, Birkh\"auser Boston, Inc., Boston, MA, 1994.

\bibitem{josefson}
{\sc B.~Josefson}, {\em On the equivalence between locally polar and globally
  polar sets for plurisubharmonic functions on {${\bf C}^{n}$}}, Ark. Mat., 16
  (1978), pp.~109--115.

\bibitem{Lucas_ex}
{\sc L.~Kaufmann}, {\em A {S}koda-type integrability theorem for singular
  {M}onge-{A}mp\`ere measures}, Michigan Math. J., 66 (2017), pp.~581--594.

\bibitem{Klimek}
{\sc M.~Klimek}, {\em Pluripotential theory}, vol.~6 of London Mathematical
  Society Monographs. New Series, The Clarendon Press, Oxford University Press,
  New York, 1991.
\newblock Oxford Science Publications.

\bibitem{Kolodziej05}
{\sc S.~Ko{\l}odziej}, {\em The complex {M}onge-{A}mp\`ere equation and
  pluripotential theory}, Mem. Amer. Math. Soc., 178 (2005), pp.~x+64.

\bibitem{NgocCuong-Kolodziej}
{\sc S.~Ko{\l}odziej and N.~C. Nguyen}, {\em Continuous solutions to
  {M}onge-{A}mp\`ere equations on {H}ermitian manifolds for measures dominated
  by capacity}.
\newblock \url{https://arxiv.org/abs/2003.05061}, 2020.

\bibitem{Moser}
{\sc J.~Moser}, {\em A sharp form of an inequality by {N}. {T}rudinger},
  Indiana Univ. Math. J., 20 (1970/71), pp.~1077--1092.

\bibitem{Hoang-MT}
{\sc V.~H. Nguyen}, {\em Improved {M}oser-{T}rudinger inequality for functions
  with mean value zero in {$\Bbb{R}^n$} and its extremal functions}, Nonlinear
  Anal., 163 (2017), pp.~127--145.

\bibitem{Okada}
{\sc M.~Okada}, {\em Espaces de {D}irichlet g\'{e}n\'{e}raux en analyse
  complexe}, J. Functional Analysis, 46 (1982), pp.~396--410.

\bibitem{Sibony_duke}
{\sc N.~Sibony}, {\em Quelques probl\`emes de prolongement de courants en
  analyse complexe}, Duke Math. J., 52 (1985), pp.~157--197.

\bibitem{Trudinger}
{\sc N.~S. Trudinger}, {\em On imbeddings into {O}rlicz spaces and some
  applications}, J. Math. Mech., 17 (1967), pp.~473--483.

\bibitem{Vigny}
{\sc G.~Vigny}, {\em Dirichlet-like space and capacity in complex analysis in
  several variables}, J. Funct. Anal., 252 (2007), pp.~247--277.

\bibitem{Vu_MA}
{\sc D.-V. Vu}, {\em Complex {M}onge--{A}mp\`ere equation for measures
  supported on real submanifolds}, Math. Ann., 372 (2018), pp.~321--367.

\bibitem{Vu_pluripolar}
\leavevmode\vrule height 2pt depth -1.6pt width 23pt, {\em Locally pluripolar
  sets are pluripolar}, Internat. J. Math., 30 (2019).

\bibitem{Vu_nonkahler_topo_degree}
\leavevmode\vrule height 2pt depth -1.6pt width 23pt, {\em Equilibrium measures
  of meromorphic self-maps on non-{K}\"{a}hler manifolds}, Trans. Amer. Math.
  Soc., 373 (2020), pp.~2229--2250.

\bibitem{Yau1978}
{\sc S.~T. Yau}, {\em On the {R}icci curvature of a compact {K}\"ahler manifold
  and the complex {M}onge-{A}mp\`ere equation. {I}}, Comm. Pure Appl. Math., 31
  (1978).

\end{thebibliography}
\bibliographystyle{siam}

\bigskip

\noindent
\Addresses
\end{document}